\documentclass[12pt,reqno]{amsart}

\usepackage{amsfonts,amsmath,amsthm}
\usepackage{amssymb,epsfig}
\usepackage{enumerate} 

\usepackage[text={425pt,650pt},centering]{geometry}
\usepackage{graphicx}
\usepackage{epsfig}
\usepackage{tikz}
\usepackage{tikz-cd}
\usepackage{caption}
\usepackage{array}
\usepackage{varwidth}
\usepackage{color} 
\definecolor{vert}{rgb}{0,0.6,0}

\numberwithin{figure}{section}

\theoremstyle{plain}
\newtheorem{thm}{Theorem}[section]

\newtheorem{ex}{Example}
\newtheorem{lem}[thm]{Lemma}
\newtheorem{cor}[thm]{Corollary}
\newtheorem{prop}[thm]{Proposition}
\theoremstyle{remark}
\newtheorem{rem}{\bf{Remark}}
\numberwithin{equation}{section}

\newcommand{\R}{\mathbb{R}}

\newcommand{\Z}{\mathbb{Z}}

\newcommand{\Lip}{{\rm Lip\,}}

\newcommand{\gam}{\gamma}
\newcommand{\del}{\delta}
\newcommand{\e}{\varepsilon }
\newcommand{\kap}{\kappa}
\newcommand{\lam}{\lambda}
\newcommand{\sig}{\sigma}

\newcommand{\Lam}{\Lambda}

\begin{document}

\title[Homogenization of Infinite-Dimensional Hamilton--Jacobi Equations]
{Periodic Homogenization of Hamilton--Jacobi Equations for Infinite Systems of Indistinguishable Particles}

\author[Seho Park]
{Seho Park}


\address[Seho Park]
{
Department of Mathematics, 
University of Wisconsin-Madison, 480 Lincoln  Drive, Madison, WI 53706, USA}
\email{park646@wisc.edu}

\date{\today}

\maketitle
\begin{abstract}
We study the homogenization of first-order Hamilton–Jacobi equations on an infinite-dimensional Hilbert space, motivated by systems of infinitely many indistinguishable particles on the torus. A central difficulty is that the analysis takes place in an infinite-dimensional setting, where the compactness arguments available in finite dimensions break down. The problem is further complicated by the possible nonconvexity of the Hamiltonian, which prevents the direct use of variational methods. Under suitable assumptions on the Hamiltonian and the initial data, we characterize the effective Hamiltonian through an associated cell problem and prove that the solutions converge to the solution of the limiting equation at the rate $O(\e ^{1/3})$. This yields a qualitative and quantitative homogenization result for a class of possibly nonconvex Hamilton–Jacobi equations in infinite dimensions.
\end{abstract}
\section{Introduction}
Hamilton–Jacobi equations in infinite-dimensional spaces arise naturally in the study of systems with infinitely many degrees of freedom, including continuum limits of interacting particle systems and dynamical models with relabeling symmetry for indistinguishable particles \cite{Gomes-Nurbekyan, Gangbo-2, feng, padhye1996fluid}. In this paper, we study the homogenization of first-order Hamilton–Jacobi equations on an infinite-dimensional Hilbert space $V = L^2(I;\R^d)$, where $I = [0,1]^d$. More precisely, we consider equations of the form

\[
    {\rm(CP)_\e } \quad
    \begin{cases}
    u^\e _t + H(\frac{x}{\e }, Du^\e ) = 0
    & \text{in } V \times (0,\infty), \\[1ex]
    u^\e (x,0)=u_0(x)
    & \text{on } V
    \end{cases}
\]
\\[0.5ex]
\noindent where $H:V \times V^{*} \rightarrow \R$ is a given Hamiltonian, $u^\e :V \times [0,\infty) \rightarrow \R$ is the unknown, and $Du^\e $ denotes the Fr\'echet derivative of $u^\e $. Since \(V\) is a Hilbert space, we identify \(V^*\) with \(V\) through the Riesz representation theorem. Our goal is to understand the asymptotic behavior of \(u^\e \) as \(\e \to 0\). A key structural feature of the problem is that, under certain assumptions, the homogenized dynamics depend only on the mean configuration. Thus the effective equation is naturally formulated on \(\R^d\), rather than on the full Hilbert space \(V\).

More precisely, for \(x\in V\), set
\[
    \mathfrak m(x):=\int_I x\,d\lambda_0\in \R^d,
    \qquad
    \iota(q):=q\chi_I\in V,
    \qquad
    Mx:=\iota(\mathfrak m(x)).
\]
Here, $Mx$ denotes the mean configuration of $x$. Then the homogenized limit $u$ will be of the form
\[
    u(x,t)=\widetilde u(\mathfrak m(x),t),
\]
where \(\widetilde u:\R^d\times[0,\infty)\to\R\) solves
\[
(\overline{\mathrm{CP}})\quad
\begin{cases}
    \widetilde u_t+\overline H(D\widetilde u)=0
        &\text{in } \R^d\times(0,\infty),\\
    \widetilde u(q,0)=\widetilde u_0(q)
        &\text{on } \R^d,
\end{cases}
\]
with
\[
    \widetilde u_0:=u_0 \circ \iota.
\]
Here \(\overline H:\R^d\to \R\) is defined through the infinite-dimensional cell problem. In our setting, the Hamiltonian is assumed to be periodic under lattice-valued translations and invariant under measure-preserving bijections.

In finite-dimensional Euclidean spaces, the theory of viscosity solutions provides a robust framework for first-order Hamilton–Jacobi equations \cite{CL83, CEL84, tran2021}. In infinite dimensions, however, the lack of local compactness makes the basic comparison, stability, and existence arguments substantially more delicate. The foundational work of Crandall and Lions established an infinite-dimensional viscosity framework for Hamilton–Jacobi equations, closely connected to ideas from differential games \cite{C-L-1, C-L-2}. In that setting, Banach spaces with the Radon--Nikodym property play an important role because smooth variational perturbation principles are available there \cite{stegall1978}, and these are central to the test-function arguments underlying the theory. Additionally, under coercivity assumptions, Ishii’s Perron-type method yields existence results for the Cauchy and Dirichlet problems \cite{Ishii}. Beyond these foundational results, a broad literature has studied infinite-dimensional Hamilton–Jacobi equations from the viewpoint of optimal control and related problems \cite{barbu, cannarsa}. These developments provide the starting point for our study of homogenization in an infinite-dimensional Hilbert space.

For homogenization problems, formal asymptotic expansions suggest that the limiting dynamics should be governed by an effective equation, while the microscopic oscillations are captured through an associated corrector problem or cell problem. In finite dimensions, this program was developed systematically by Lions, Papanicolaou, and Varadhan \cite{LPV}, and further advanced by Evans through the perturbed test function method \cite{Evans-perturb}. Since then, both qualitative and quantitative aspects of homogenization have been extensively developed \cite{Con96, Con97, capuzzo-rate, mitake2019rate,  TY, han2023rate}. In infinite-dimensional settings, by contrast, the literature remains comparatively limited. Existing results on cell problems and effective dynamics are often tied to convexity, variational structure, or methods from weak KAM theory \cite{fathi2003}.
In particular, Gomes and Nurbekyan study the cell problem in an infinite-dimensional framework \cite{Gomes-Nurbekyan}. This work is closely related to their work on minimizers of variational problems in Hilbert spaces~\cite{gomes2015} and builds on earlier work of Gangbo and Tudorascu on infinite-dimensional weak KAM theory and action-minimizing trajectories in Lagrangian systems~\cite{Gangbo-1, Gangbo-2}. Our problem is also motivated by infinite systems of indistinguishable particles on the $d$-dimensional torus, but we allow general, possibly nonconvex Hamiltonians for which the connection to an underlying Lagrangian or weak KAM structure is not directly available \cite{capuzzo-rate, QTY, TY}. Related homogenization problems and their quantitative behavior have also been studied in Wasserstein space by Ding, Ekren, Han, and Zitridis \cite{Ding-Ekren-Han-Zitridis}. We also note the related viewpoint of Feng \cite{feng}, who studies Hamilton–Jacobi equations in the space of probability measures arising from the hydrodynamic limit of $N$-particle dynamics through a multiscale convergence approach. 

\subsection{Motivation.} 
In this paper, we adopt the random-variable framework of Gomes and Nurbekyan \cite{Gomes-Nurbekyan} for a mechanical system of infinitely many indistinguishable particles on the $d$-dimensional torus ($\mathbb{T}^d$). It is useful to view this setting as the infinite-dimensional analogue of a finite system of indistinguishable particles on the torus. For a system of $N$ particles in $\mathbb{T}^d$, the configuration is described by a point in $(\mathbb{T}^d)^N$, and indistinguishability means that configurations differing only by a permutation of particle labels should be identified. In the infinite-particle setting, we parametrize particle labels by points of \(I=[0,1]^d\). The configuration of the system is encoded by a map $x \in L^2(I; \R^d)$. For each point $i \in I$, $x(i) \in \R^d$ represents the position of the particle labeled by $i$. Equivalently, $x$ can be viewed as an $\R^d$-valued random variable on $I$. Let 
\begin{equation*}
    V:= L^2(I; \R^d)
\end{equation*}
be the corresponding Hilbert space of configurations. To encode the periodicity, we consider the additive subgroup
\begin{align*}
    \Lam := L^2(I; \Z^d) = \{z \in V : z(i) \in \Z^d,\, \lambda_0\text{-a.e. } i \in I\}.
\end{align*}
Here, $\lam_0$ denotes the Lebesgue measure on $I$. This subgroup plays the role of lattice translations, and the quotient $V/\Lam$ is viewed as the infinite-dimensional torus. 

Next, let $\mathcal{G}$ be the set of all bijections $g: I\to I$ such that $g$ and $g^{-1}$ are Borel measurable and both push $\lam_0$ forward to itself. Namely,
\begin{multline*}
    \mathcal{G} :=
    \left\{
    g:I\to I \,;\,
    g \text{ is bijective},\ 
    g,\ g^{-1} \text{ are Borel measurable},\ 
    g_{\#}\lambda_0= g^{-1}_\#\lambda_0 = \lam_0
    \right\}.
\end{multline*}
In other words, $\mathcal G$ is the group of measure-preserving rearrangements of $I$, and it plays the role of the permutation group in the infinite-dimensional setting. For $x\in V$ and $g\in\mathcal G$, the composition $x\circ g$ represents the same particle configuration as $x$. The quotient $(V/\Lam)/\mathcal{G}$ is viewed as an infinite-dimensional symmetric torus. From this viewpoint, the assumptions of periodicity and rearrangement invariance are the natural infinite-dimensional analogues of the symmetries in the finite-particle problem. 

We further define the metric $d_{SS^d}: V\times V \to \R$, the equivalence relation $\sim$, and the quotient $SS^d$ as
\begin{equation*}
    d_{SS^d}(x,y):= \inf_{g\in G, z\in\Lambda} \|x-y\circ g-z\|_{L^2},\qquad
    x\sim y\Leftrightarrow d_{SS^d}(x,y)=0, \qquad
    SS^d:=V/{\sim}.
\end{equation*}
\noindent It is known \cite{Gomes-Nurbekyan} that $(SS^d,d_{SS^d})$ is isometric to the Wasserstein space $(\mathcal{P}(\mathbb{T}^d), W_2)$ and is in particular a compact, complete, separable metric space. This compactness is essential in our analysis to compensate for the lack of local compactness of $V$.

\subsection{Main Assumptions}
The dynamics of an infinite-particle system can be described by the Hamiltonian $H:V\times V \to \R$. Throughout the paper, we assume that $H$ satisfies \rm{(H1)}--\rm{(H4)}. Conditions \rm{(H1)}--\rm{(H2)} hold for all $x,p\in V$, $z\in\Lambda$, and $g\in G$. In \rm{(H3)}, there exists $L_x>0$ such that the first estimate holds for all $x,y,p\in V$, and, for every $R>0$, there exists $L_R>0$ such that the second estimate holds for all $x\in V$ and $p,q\in B_V(0,R)$.
\begin{align*}
    &(\mathrm{H1})\quad H(x + z, p) = H(x,p) &&\text{(Periodicity)} \\
    &(\mathrm{H2})\quad H(x \circ g, p\circ g) = H(x,p) &&\text{(Rearrangement Invariance)} \\
    &(\mathrm{H3})
    \begin{cases}
        |H(x,p) - H(y,p)| \le L_x(1+ \|p\|)\|x-y\|\\
        |H(x,p) - H(x,q)| \le L_R \|p-q\|, \quad p,q \in B(0,R)
    \end{cases} \quad &&\text{(Locally Lipschitz)}\\
    &(\mathrm{H4})\quad \lim_{\|p\| \to \infty} \inf_{x \in V} H(x,p) = +\infty \;&&\text{(Coercivity)}
\end{align*}

\noindent Condition $(\mathrm{H1})$ expresses invariance under lattice-valued translations, while $(\mathrm{H2})$ expresses invariance under relabeling of indistinguishable particles. $(\mathrm{H3})$ is imposed so that the comparison and existence results of Crandall and Lions \cite{C-L-1, C-L-2} apply to both the Cauchy problem and the discounted cell problem. Assumptions $(\mathrm{H3})$ and $(\mathrm{H4})$ are used to construct barriers and obtain uniform Lipschitz estimates. 

For the initial data, we assume:
\begin{align*}
(\mathrm{I1})\quad & u_0(x) = u_0(Mx)\\
(\mathrm{I2})\quad & u_0 \in C_b^1(V)
\end{align*}

\noindent 
Here, $C_b^1(V)$ denotes the space of bounded and continuously Fr\'echet differentiable functions whose derivatives are also bounded. Note that the condition $(\mathrm{I1})$ is stronger than rearrangement invariance. This additional assumption is essential for obtaining the finite-dimensional homogenized limit.

\subsection{Main Results}
In this paper, we establish the homogenization of $\rm(CP)_\e $ under the assumptions above. 

\begin{thm}\label{thm1}
    Assume that \(H\) satisfies \((\mathrm{H}1)\)--\((\mathrm{H}4)\) and that \(u_0\) satisfies \((\mathrm{I}1)\)--\((\mathrm{I}2)\). For each \(\e >0\), let \(u^\e \) be the unique bounded uniformly continuous viscosity solution of \((\mathrm{CP})_\e \). Then there exists a continuous coercive effective Hamiltonian
    \[
        \overline H:\R^d\to\R
    \]
    such that
    \[
        u^\e (x,t)\longrightarrow \widetilde u(\mathfrak m(x),t)
    \]
    uniformly on bounded subsets of \(V\times[0,\infty)\) and \(\widetilde u\) is the unique viscosity solution of $\mathrm{(\overline{CP})}$.
\end{thm}

\begin{rem}[Model reduction]\label{rem1}
A key consequence of Theorem~1.1 is the reduction of the limiting problem from an infinite-dimensional equation to a finite-dimensional one. Although the oscillatory equations \(\mathrm{(CP)_\e }\) are posed on \(V=L^2(I;\R^d)\), the homogenized solution sees only the mean configuration. 
This model-reduction phenomenon is not apparent at the level of \(\mathrm{(CP)_\e }\), where both the equation and the cell problem are formulated on the infinite-dimensional space. It is one of the main structural features of the result.
\end{rem}

We first establish compactness and convergence of \(u^\e \) as \(\e \to0\) in Section~\ref{sec2}. We then characterize the effective Hamiltonian $\overline{H}$ via an associated cell problem and show that $u$ solves the effective equation in Section~\ref{sec3}. Our approach combines the viscosity framework in infinite dimensions with the construction of barriers in Perron’s method, the analysis of the associated cell problem, and an adaptation of the perturbed test function method that does not rely on the convexity of $H$. The preceding remark highlights an essential point: the limit is governed only by the mean configuration, so the effective equation is finite-dimensional even though the original oscillatory problem is genuinely infinite-dimensional.

Our second result provides a quantitative convergence rate.
\begin{thm}\label{thm2}
    Under the assumptions of Theorem 1.1, for every \(T>0\) there exists \(C_T>0\), independent of \(\e \), such that
    \[
        \sup_{(x,t)\in V\times[0,T]}
        \left|u^\e (x,t)-\widetilde u(\mathfrak m(x),t)\right|
        \le C_T\e ^{1/3}.
    \]
\end{thm}

In Section~\ref{sec4}, we develop the quantitative argument that extends the finite-dimensional rate estimates \cite{capuzzo-rate} to the present infinite-dimensional, possibly nonconvex setting. A distinct feature of Theorem~\ref{thm2} is that it compares an infinite-dimensional object, \(u^\e\), with the effectively finite-dimensional limit \(u\). Therefore, although the estimate in Theorem~\ref{thm2} formally resembles the finite-dimensional rate estimate, its proof requires additional care: the oscillatory solution, the cell problem, and the correctors are posed on the full Hilbert space, while the limiting dynamics live only on the finite-dimensional mean variable. 

Together, these results establish qualitative and quantitative homogenization for possibly nonconvex Hamilton–Jacobi equations on the infinite-dimensional Hilbert space with the structural symmetries considered here. The analysis must overcome the lack of local compactness and the absence of a direct variational structure under the restrictive assumption that the initial data depend only on the mean configuration.

\begin{ex}[A mean-field Hamiltonian]
Let \(V_0,W_0\in W^{1,\infty}(\mathbb T^d)\), and view them as \(\mathbb Z^d\)-periodic functions on \(\R^d\). For \(x,p\in V=L^2(I;\R^d)\), define
\[
\mathcal V(x)
:=
\int_I V_0(x(i))\,d\lambda_0(i)
+
\frac12\int_I\int_I W_0(x(i)-x(j))\,d\lambda_0(i)\,d\lambda_0(j),
\]
and set
\[
        H(x,p):=\frac12\|p\|_{L^2}^2+\mathcal V(x).
\]
\end{ex}
This Hamiltonian represents the sum of kinetic energy, external potential energy ($V_0$) and mean-field pairwise interaction energy ($W_0$). Since \(V_0\) and \(W_0\) are periodic on \(\mathbb T^d\), we have $\mathrm{(H1)}$.
Moreover, since the integrals defining \(\mathcal V(x)\) depend only on the distribution of the particle configuration \(x\), and since \(\|p\circ g\|_{L^2}=\|p\|_{L^2}\), the Hamiltonian is invariant under measure-preserving rearrangements $\mathrm{(H2)}$.
The bounded Lipschitz assumptions on \(V_0\) and \(W_0\) imply that \(\mathcal V\) is Lipschitz on \(V\). Hence $\mathrm{(H3)}$ holds. Finally, since \(\mathcal V\) is bounded, the quadratic term in $H$ yields the coercivity condition $\mathrm{(H4)}$. Therefore \(H\) satisfies \(\mathrm{(H1)}\)--\(\mathrm{(H4)}\).

\section{Well-posedness, regularity, and compactness}\label{sec2}

We first record the infinite-dimensional viscosity framework used in this paper. We use the Radon--Nikodym property through Stegall's variational principle~\cite{stegall1978}. A Banach space $V$ is said to have the Radon--Nikodym property if, for every bounded lower semicontinuous function $\varphi$ on a closed ball $B \subset V$ and every $\e > 0$, there exists $p \in V^*$ with $\|p\| \leq \e$ such that $\varphi +p$ attains its minimum on $B$. The corresponding maximum statement follows by changing signs.  Since $V=L^2(I;\mathbb R^d)$ is a Hilbert space, both $V$ and every finite product of copies of $V$ and $\mathbb R$ have this property.

\subsection{Well-posedness}
We use the comparison and existence theory of Crandall and Lions for Hamilton--Jacobi equations in Banach spaces with the Radon--Nikodym property \cite{C-L-1, C-L-2}. In the present Hilbert-space setting, we take
\[
d(x,y)=\|x-y\|,
\qquad
\nu(x)=\sqrt{1+\|x\|^2}
\]
for the functions appearing in \cite{C-L-1,C-L-2}. Note that the derivatives of \(d\), away from the diagonal, and of \(\nu\) are uniformly bounded. Moreover, for \(x\neq y\) and \(\ell\geq0\),
\begin{equation}\label{eq:s2_1}
    H\left(y,\ell\frac{x-y}{\|x-y\|}\right)
    -
    H\left(x,\ell\frac{x-y}{\|x-y\|}\right)
    \leq
    L_x(1+\ell)\|x-y\|.
\end{equation}
Together with \rm{(H3)}, \eqref{eq:s2_1} verifies the assumption needed in \cite{C-L-1, C-L-2}.

For each fixed \(\varepsilon>0\), the Hamiltonian
\[
    H^\varepsilon(x,p):=H(x/\varepsilon,p)
\]
satisfies the same structural conditions, with spatial Lipschitz constant \(L_x/\varepsilon\).  We therefore obtain the following proposition from the theory in \cite{C-L-1, C-L-2}.

\begin{prop}\label{prop:wellposedness}
Let $T>0$.
\begin{enumerate}
\item For every $\varepsilon>0$ and every $g\in\operatorname{BUC}(V)$, the
Cauchy problem
\begin{equation}\label{eq:s2_2}
\begin{cases}
w_t+H(x/\varepsilon,Dw)=0 & \text{in }V\times(0,T),\\
w(x,0)=g(x) & \text{on }V
\end{cases}
\end{equation}
has a unique viscosity solution $w\in\operatorname{BUC}(V\times[0,T])$.
\item Let $w_1,w_2$ be the viscosity solutions of \eqref{eq:s2_2} with initial data $g_1,g_2$, respectively. Then
\[
\sup_{V\times[0,T]}|w_1-w_2|
\le \|g_1-g_2\|_{L^\infty(V)}.
\]
\item For every $\lambda>0$ and $p\in\mathbb R^d$, the discounted cell
problem
\[
\lambda v+H(y,p\chi_I+Dv)=0\qquad\text{in }V
\]
has a unique solution $v\in\operatorname{BUC}(V)$.
\end{enumerate}
\end{prop}

\subsection{Uniform Lipschitz estimates}
We next establish a uniform Lipschitz estimate using the barrier argument from Perron's method \cite{Ishii}. We first record a lemma giving a uniform bound for $H$ on bounded subsets of the momentum variable.

\begin{lem}
\label{lem:H-bounded-on-bounded-p}
For every $R>0$,
\[
    C_R:=\sup_{x\in V, \|p\|\le R}|H(x,p)|<\infty.
\]
\end{lem}

\begin{proof}
Let $\mathcal N:\mathbb R^d\to\mathbb Z^d$ be a Borel nearest-integer map, chosen coordinatewise with a fixed convention at ties. Set $z_x :=\mathcal N \circ x$.  Then $z_x\in\Lambda$ and
\[
    \|x-z_x\|_{L^2}\le \sqrt d / 2.
\]
By periodicity and \rm{(H3)}, for $\|p\|\le R$,
\begin{align*}
|H(x,p)|
=|H(x-z_x,p)| \le |H(0,0)|+L_RR+L_x(1+R)(\sqrt d / 2).
\end{align*}
\end{proof}

Set
\[
R_0:=\|Du_0\|_{L^\infty(V)}
\qquad\text{and}\qquad
C_0:=
\sup_{y\in V, \|p\|\leq R_0} |H(y,p)|.
\]
This constant is finite by Lemma~\ref{lem:H-bounded-on-bounded-p} and
\rm{(I2)}.  The functions
\[
    \underline u(x,t)=u_0(x)-C_0t,
    \qquad
    \overline u(x,t)=u_0(x)+C_0t
\]
are, respectively, a viscosity subsolution and supersolution of $\mathrm{(CP)_\e}$. Hence the comparison principle gives
\begin{equation}\label{eq:s2_3}
 u_0(x)-C_0t\le u^\varepsilon(x,t)\le u_0(x)+C_0t.
\end{equation}

The uniform Lipschitz continuity of $u^\e$ in time will follow from~\eqref{eq:s2_3}, the comparison principle and the semigroup property. See the details in Proposition~\ref{prop:uniform-lipschitz}. We next establish its uniform Lipschitz continuity in the spatial variable through the following steps. We first record the Lipschitz regularity of viscosity subsolutions of the eikonal equation.

\begin{lem}
\label{lem:eikonal-lipschitz}
Let $w\in\operatorname{BUC}(V)$ be a viscosity subsolution of
\[
\|Dw\|-L\le0\qquad\text{in }V.
\]
Then $w$ is $L$-Lipschitz on $V$.
\end{lem}

\begin{proof}
Suppose, to the contrary, that $w(x)-w(y)>L\|x-y\|$ for some \(x,y\in V\). Choose \(L'>L\) such that
\[
    w(x)-w(y)>L'\|x-y\|.
\]
Define
\[
    \Phi(z):=w(z)-w(y)-L'\|z-y\|, \qquad z \in V.
\]
Then $\Phi(x)>0$ and $\Phi(y)=0$. Since \(w\) is bounded, we may choose \(R>\|x-y\|\) and \(\sigma>0\) such that
\[
    \Phi(x)\geq 2\sigma \qquad \text{and}\qquad \sup_{\|z-y\|=R}\Phi(z)\leq -2\sigma.
\]

Apply Stegall's variational principle to \(\Phi\) on the closed ball \(\overline{B(y,R)}\). For every sufficiently small \(\eta>0\), there exists \(p_\eta\in V\) with $\|p_\eta\|<\eta$ such that
\begin{equation}\label{eq: s2_4}
    z\longmapsto
    \Phi(z)-\langle p_\eta,z-y\rangle
\end{equation}
attains its maximum on \(\overline {B(y,R)}\) at some point \(z_\eta\).
Choose \(\eta>0\) small enough that $\eta R<\sigma$. Then, for every \(z\in\partial B(y,R)\), we have
\[
\Phi(z)-\langle p_\eta,z-y\rangle \leq
-2\sigma+\|p_\eta\|R\ \le -\sigma.
\]
On the other hand,
\begin{equation}\label{eq: s2_5}
    \Phi(x)-\langle p_\eta,x-y\rangle \geq
    2\sigma-\|p_\eta\|\|x-y\|\ \ge \sigma.
\end{equation}
It follows that $z_\eta\in B(y,R)$. Moreover \eqref{eq: s2_5} implies that $z_\eta\neq y$ because the value of~\eqref{eq: s2_4} at \(y\) is zero.

Consequently, the function
\[
    \phi_\eta(z) :=
    w(y)+L'\|z-y\|
    +\langle p_\eta,z-y\rangle+C_\eta,
\]
where \(C_\eta\) is chosen so that \(\phi_\eta(z_\eta)=w(z_\eta)\), touches \(w\) from above at the interior point \(z_\eta\). Since \(z_\eta\neq y\), $\phi_\eta$ is differentiable at \(z_\eta\), and
\[
    D\phi_\eta(z_\eta) = 
    L'\frac{z_\eta-y}{\|z_\eta-y\|}
    +p_\eta, \qquad \|D\phi_\eta(z_\eta)\| \ge L'-\|p_\eta\|.
\]
Choosing \(\eta<L'-L\), we obtain
\[
    \|D\phi_\eta(z_\eta)\|>L,
\]
which contradicts the viscosity subsolution inequality. 
\end{proof}
We next record a time-slice property for viscosity subsolutions.

\begin{lem}\label{lem:time-slice}
Let $F:V\times V\to\mathbb R$ be continuous, and let $w\in\operatorname{BUC}(V\times[0,T])$ be a viscosity subsolution of
\[
    w_t+F(x,Dw)=0\qquad\text{in }V\times(0,T).
\]
Assume that $w$ is $C_t$-Lipschitz in time, uniformly in $x$.  Then, for every $t_0\in(0,T)$, the map $x\mapsto w(x,t_0)$ is a viscosity subsolution
of
\[
    F(x,Dw)\le 2C_t\qquad\text{in }V.
\]
\end{lem}
\begin{proof}
Let $\psi\in C^1(V)$ touch $w(\cdot,t_0)$ from above at $x_0$. Replacing $\psi$ by
$$
    \psi_\rho(x):=\psi(x)+\rho\|x-x_0\|^2
$$
for some $\rho>0$, we may assume that
$$
    w(x,t_0)-\psi_\rho(x)
    \leq
    w(x_0,t_0)-\psi_\rho(x_0)
    -\rho\|x-x_0\|^2
$$
on a closed ball $\overline {B(x_0,r)}$. Note that $D\psi_\rho(x_0)=D\psi(x_0)$.

Choose $\tau>0$ such that $[t_0-\tau,t_0+\tau]\subset(0,T)$ and $C_t\tau\leq \rho r^2 / 8$. Let
$$
    E:= \left\{
    (x,t):
    \frac{\|x-x_0\|^2}{r^2} +
    \frac{\|t-t_0\|^2}{\tau^2} \leq 1
    \right\}.
$$
For $a>0$, define
$$
    \Phi_a(x,t) :=
    w(x,t)-\psi_\rho(x)-a|t-t_0|^2.
$$
If $(x,t)\in\partial E$, then either $\|x-x_0\|\ge r/2$ or $|t-t_0|\ge \tau/2$. In the first case,
$$
    \Phi_a(x,t) \leq
    \Phi_a(x_0,t_0) - \rho r^2/4 + C_t\tau \leq
    \Phi_a(x_0,t_0)-\rho r^2/8.
$$
In the second case,
$$
    \Phi_a(x,t) \leq
    \Phi_a(x_0,t_0) +C_t\tau- a\tau^2 / 4.
$$
Hence, for all sufficiently large $a$, there exists $\sigma>0$ such that
\begin{equation}\label{eq:s2_6}
    \sup_{\partial E}\Phi_a
    \leq
    \Phi_a(x_0,t_0)-\sigma.
\end{equation}

By Stegall's variational principle, there exist
$$
    (p_a,k_a)\in V\times\mathbb R,
    \qquad \|p_a\|\le\frac{\sig}{8r(a+1)} , \qquad|k_a|\le\frac{\sig}{8\tau (a+1)} ,
$$
such that
$$
    (x,t)\longmapsto
    \Phi_a(x,t)
    -\langle p_a,x-x_0\rangle
    -k_a(t-t_0)
$$
attains its maximum on $E$ at some point $(x_a,t_a)$. Since the absolute value of the perturbation on $E$ is less than $\sigma/2$, \eqref{eq:s2_6} implies
$$
    (x_a,t_a)\in\operatorname{int}E.
$$

The viscosity subsolution inequality at $(x_a,t_a)$ gives
$$
    2a(t_a-t_0)+k_a +
    F\bigl(x_a,D\psi_\rho(x_a)+p_a\bigr) \leq0.
$$
Comparing the maximum value at $(x_a,t_a)$ with the value at $(x_a,t_0)$ yields
$$
    a|t_a-t_0|^2 \leq
    C_t|t_a-t_0| + |k_a||t_a-t_0|.
$$
Therefore,
$$
    a|t_a-t_0| \leq C_t+|k_a|,
$$
and hence $t_a\to t_0$ as $a\to\infty$. Similarly, maximality yields $ x_a\rightarrow x_0$ as $a \to \infty$.

Consequently,
$$
    F\bigl(x_a,D\psi_\rho(x_a)+p_a\bigr) \leq
    2a|t_a-t_0|+|k_a| \leq
    2C_t+3|k_a|.
$$
Letting $a\to\infty$ gives
$$
    F(x_0,D\psi(x_0))\leq2C_t.
$$
This proves the claim.
\end{proof}

We now prove the uniform Lipschitz regularity for $u^\e$.
\begin{prop}
\label{prop:uniform-lipschitz}
There exists $C>0$, independent of $\varepsilon$, such that
\[
|u^\varepsilon(x,t)-u^\varepsilon(y,s)|
\le C\bigl(\|x-y\|+|t-s|\bigr)
\]
for all $x,y\in V$ and $s,t\ge0$.
\end{prop}

\begin{proof}
By \eqref{eq:s2_3} and Proposition~\ref{prop:wellposedness}, for every $h,t\ge0$,
\begin{align*}
    \|u^\varepsilon(\cdot,t+h)-u^\varepsilon(\cdot,t)\|_\infty
    \le \|u^\varepsilon(\cdot,h)-u_0\|_\infty\le C_0h.
\end{align*}
Thus $u^\varepsilon$ is $C_0$-Lipschitz in time.  By Lemma~\ref{lem:time-slice}, every positive-time slice is a viscosity subsolution of
\[
    H(x/\varepsilon,Du^\varepsilon)\le 2C_0.
\]
Coercivity assumption \rm{(H4)} gives $L_0>0$, independent of $\varepsilon$, such that
\[
    H(x,p)>2C_0\qquad\text{whenever }\|p\|>L_0.
\]
Hence every time slice is a viscosity subsolution of
$$
    \|Du^\varepsilon\|-L_0\le0.
$$
Lemma~\ref{lem:eikonal-lipschitz} gives the uniform spatial Lipschitz estimate.  At $t=0$, the same conclusion follows from \rm{(I2)}.
\end{proof}

\subsection{Convergence of $u^\e $ as $\e  \rightarrow 0^+$}
We now exploit the symmetry of the problem to obtain compactness and to identify the effective variable. The key point is that the assumptions on the Hamiltonian and the initial data single out the mean configuration as the only relevant quantity in the limit. 

Let
\[
    Y := \left\{ y \in V \,;\,\mathfrak m(y) = 0\right\}.
\]
Then $Y$ is a closed subspace of $V$. The orthogonal decomposition
\[
    V = Y \oplus Y^\perp
\]
plays a central role in what follows. Indeed, it separates the mean-zero directions from the finite-dimensional space of constant configurations $Y^\perp$. This is precisely the decomposition that allows us to use both the symmetry of $H$ and the assumption that the initial data depend only on the mean configuration. The oscillatory problem is posed on the whole infinite-dimensional space $V$, but the symmetry forces the solution to become insensitive to the $Y$-directions, so that only the $Y^\perp$-component survives in the limit.

\begin{lem}\label{lem2.6}
$V = Y \oplus Y^\perp$ and $Y^\perp = \{c \chi_I \,;\, c \in \R^d \}$.
\end{lem}
\begin{proof}
    Let
    \[
        K:=\{c\chi_I : c\in \R^d\}.
    \]
    It is immediate that $K\subset Y^\perp$, because $\langle c\chi_I,y\rangle = 0$ for any $y\in Y$ and $c \in\R^d$.
    
    Conversely, let $f\in Y^\perp$. Write
    \[
    f = (f-Mf)+Mf.
    \]
    Since $\int_I (f-Mf)\,d\lambda_0=0$, we have $f-Mf\in Y$. Because $f, Mf\in Y^\perp$,
    \[
    0=\langle f,f-Mf\rangle
    =\|f-Mf\|^2.
    \]
    Hence $f=Mf\in K$. Therefore $Y^\perp=K$.
\end{proof}
For each $\e  > 0$, we next introduce the additive subgroup of $Y$
\[
    Y_\e  := \{ \e  z \,;\, z \in \Lambda \cap Y \}.
\]
Later, we will show that the periodicity of the Hamiltonian implies the invariance of $u^\e$ under translations by elements of $Y_\e $. The next decomposition makes this invariance effective. It shows that every configuration $x \in V$ can be written as the sum of three parts: its mean component $Mx \in Y^\perp$, an $\e$-grid-valued mean-zero part in $Y_\e $, and a remainder in $Y$ whose size is $O(\e )$. In this way, the $Y_\e $-invariance can be combined with the Lipschitz regularity of $u^\e $ to show that $u^\e (x,t)$ is close to $u^\e (Mx,t)$, uniformly in $(x,t)$.

\begin{lem}\label{lem:epsilon-decomposition}
For every $x\in V$ and every $\e >0$, there exist $x_\e \in Y$ and $z_x\in \Lambda\cap Y$ such that
\[
x = x_\e  + \e  z_x + Mx,
\qquad
\|x_\e \|_{L^2}\le C\e ,
\]
where $C>0$ depends only on the dimension $d$.
\end{lem}
\begin{proof} 
Write
\[
    x=P_Yx+Mx,
\]
where \(P_Y\) is the orthogonal projection onto \(Y\). Choose a measurable nearest-integer map \(\mathcal N:\R^d\to\mathbb Z^d\), with a fixed convention at ties, and define
\[
    z_0(i):= \mathcal N(\e ^{-1}P_Yx(i))
\]
for every $i \in I$. Then \(z_0\in\Lambda\) and
\[
    \|P_Yx-\e  z_0\|_{L^2}\le \e\sqrt d/2 .
\]
Set
\[
    m = (m_1,\dots ,m_d):= \mathfrak m(z_0).
\]
Note that \(m \in [-1/2,1/2]^d\). For each \(j\), choose a measurable set \(A_j\subset I\) such that
\[
    \lambda_0(A_j)=|m_j|, \qquad j = 1,\dots ,d.
\]
Define
\[
    \eta_j:=\operatorname{sgn}(m_j)\mathbf 1_{A_j},
    \qquad
    \eta:=(\eta_1,\dots,\eta_d).
\]
Then \(\eta\in\Lambda\) satisfies
\[
    \mathfrak m (\eta) = m, \qquad \text{and} \qquad \|\eta\|_{L^2}^2 \le d/2.
\]
Now set
\[
    z_x:=z_0-\eta\in \Lambda\cap Y,
    \qquad
    x_\e :=P_Yx-\e  z_x.
\]
Then
\[
    x=x_\e +\e  z_x+Mx
\]
and
\[
    \|x_\e \|_{L^2}
    \le
    \|P_Yx-\e  z_0\|_{L^2}
    +
    \e \|\eta\|_{L^2}
    \le C\e ,
\]
where \(C\) depends only on \(d\).
\end{proof}

This decomposition is the key structural step in the argument. In the spatial variable of $u^\e$, the term $\e  z_x$ can be removed using the $Y_\e $-invariance, while the remaining error $x_\e $ is small in $L^2$ and therefore contributes only an $O(\e )$ error by the equi-Lipschitz estimate established earlier. Consequently, the solution $u^\e $ is uniformly close to a function depending only on the mean variable $Mx$. We now show that this mechanism indeed yields the required invariance and reduction.

\begin{prop}\label{prop:compactness}
For every $\e >0$, the solution $u^\e $ of $\mathrm{(CP)_\e }$ is $Y_\e $-invariant, that is,
\[
u^\e (x+\e  z,t)=u^\e (x,t)
\qquad
\text{for all } z\in \Lambda\cap Y \text{ and } (x,t)\in V\times [0,\infty).
\]
Moreover, there exists a constant $C>0$, independent of $\e $, such that
\[
|u^\e (x,t)-u^\e (Mx,t)|\le C\e 
\qquad
\text{for all } (x,t)\in V\times [0,\infty).
\]
Consequently, up to a subsequence,
\[
    \sup_{\|x\|\le R,\ 0\le t\le T}
    |u^\varepsilon(x,t)-u(x,t)|\to0
\]
for every \(R,T>0\). Also, the limit $u(x,t)$ depends only on $Mx$ and time.
\end{prop}

\begin{proof}
    We first show that $u^\e $ is $Y_\e $-invariant. Fix $z\in \Lambda\cap Y$. Consider
    \[
        v^\e (x,t):=u^\e (x+\e  z,t).
    \]
    Then
    \[
        v^\e _t(x,t)+H\!\left(\frac{x+\e  z}{\e },Dv^\e (x,t)\right)=0.
    \]
    By $\Lambda$-periodicity of $H$,
    \[
    H\!\left(\frac{x+\e  z}{\e },p\right)=H\!\left(\frac{x}{\e }+z,p\right)=H\!\left(\frac{x}{\e },p\right).
    \]
    Also, since $z\in Y$, we have $M(x+\e  z)=Mx$, and therefore, by assumption $\textrm{(I1)}$,
    \[
    u_0(x+\e  z)=u_0(x).
    \]
    Hence $v^\e $ is a viscosity solution of $\mathrm{(CP)_\e }$. By uniqueness,
    \[
        u^\e (x+\e  z,t)=u^\e (x,t),
    \]
    which proves the $Y_\e $-invariance of $u^\e$.
    
    Next, let $x\in V$. By Lemma~\ref{lem:epsilon-decomposition}, we may write
    \[
        x=x_\e +\e  z_x +Mx
    \]
    with $z_x\in \Lambda\cap Y$ and $\|x_\e \|_{L^2}\le C\e $. Using the $Y_\e $-invariance and Proposition~\ref{prop:uniform-lipschitz},
    \[
    |u^\e (x,t)-u^\e (Mx,t)|
    =
    |u^\e (x_\e +Mx,t)-u^\e (Mx,t)|
    \le C\|x_\e \|_{L^2}
    \le C\e .
    \]
    Thus $u^\e $ is uniformly $O(\e )$-close to its restriction to $Y^\perp\cong \R^d$.
    
    Since the family $\{u^\e \}_{\e >0}$ is equi-Lipschitz, the restrictions $u^\e |_{Y^\perp\times [0,\infty)}$
    form an equibounded and equicontinuous family on each compact subset of $Y^\perp \times [0,\infty)$. By the Arzel\`a--Ascoli theorem,
    along a subsequence,
    \[
    u^\e |_{Y^\perp\times [0,\infty)} \to u
    \quad\text{as } \e  \to 0^+ \text{ locally uniformly on } Y^\perp \times [0,\infty),
    \]
    for some continuous function $u$ on $Y^\perp \times [0,\infty)$. We then extend $u$ to \(V\times[0,\infty)\) by  
    \[
    u(x,t)=u(Mx,t) \quad \text{for every }(x,t) \in V \times [0,\infty).
    \]
    The estimate above implies that $u^\e  \to u$ uniformly on bounded subsets of \(V\times[0,\infty)\) and the limit depends on the spatial variable only through the mean.
\end{proof}

\begin{cor} Let \(u\) be a locally uniform subsequential limit of \(\{u^\varepsilon\}_{\varepsilon>0}\). Then there exists $\widetilde u:\mathbb R^d\times[0,\infty)\to\mathbb R$ such that
\[
u(x,t)=\widetilde  u(\mathfrak m(x),t).
\]
Moreover, if \(C\) is the constant in Proposition~\ref{prop:compactness}, then
\[
|\widetilde u(q,t) - \widetilde u(r,s)|
\leq
C\bigl(|q-r|+|t-s|\bigr)
\]
for every \(q,r\in\mathbb R^d\) and \(s,t\geq0\).
\end{cor}

Proposition~\ref{prop:compactness} is the main consequence of the symmetry structure developed in this section. Although $\mathrm{(CP)_\e }$ is posed on the infinite-dimensional Hilbert space $V$, the structural symmetries and the assumption on $u_0$ force the solution to be uniformly $O(\e )$-close to its restriction to the mean-configuration space $Y^\perp  \cong \mathbb{R}^d$. This explains why the homogenized problem is ultimately governed by finite-dimensional effective dynamics.

Note that Proposition ~\ref{prop:compactness} yields subsequential convergence. Once the effective equation is identified and uniqueness is established, the full convergence of the whole family $\{u^\e \}_{\e >0}$ follows.

\section{The cell problem and the effective equation} \label{sec3}

In this section, we identify the effective Hamiltonian through an associated cell problem and then prove that the locally uniform limit of $u^\e $ satisfies the effective equation. The overall strategy is analogous to the periodic homogenization theory in finite dimensions: one first solves the cell problem for each macroscopic momentum, then defines the effective Hamiltonian from the corresponding ergodic constant, and finally applies a perturbed test function argument in the spirit of Evans \cite{Evans-perturb}. The main additional difficulty in the present infinite-dimensional setting is that the microscopic variable lives in $V$, where local compactness is no longer available. In particular, when carrying out the perturbed test function argument, one cannot rely on the usual finite-dimensional maximization procedure. Instead, we use the Radon--Nikodym property through Stegall's theorem to produce approximate maximizers. At the same time, to recover compactness on the microscopic side, we pass from $V$ to the compact quotient space $SS^d$. Thus, both the Radon--Nikodym property and the compactness of $SS^d$ play essential roles in the analysis below.

\subsection{The cell problem}
We first explain why the cell problem is naturally parameterized only by momenta in \(Y^\perp\). Recall from Proposition ~\ref{prop:compactness} that every subsequential limit $u$ of $u^\e $ depends only on the mean variable $Mx$. Hence, if a test function touches the limit from above or below, its spatial derivative vanishes in the mean-zero directions \(Y\), and therefore belongs to \(Y^\perp\). This restriction is also consistent with the geometric interpretation of the cell problem in the weak KAM formulation. On the infinite-dimensional symmetric torus, $V/\Lam/\mathcal{G}$, closed one-forms are represented by
\[
    DU + p\chi_I, \qquad p\in\R^d,
\]
where \(DU\) is the exact part and \(p\chi_I\in Y^\perp\) is the constant representative of the non-exact part~\cite{Gomes-Nurbekyan}. Consequently, the macroscopic momentum parameter in the cell problem should be
taken from \(Y^\perp\).
 
For $p \in \R^d$, the associated cell problem is 
\begin{equation}\label{eq:cell-problem}
    H(y, p \chi_I + Dv(y)) = c \quad\text{in }V,
\end{equation}
where $c\in\R$ is an unknown constant. As in the finite-dimensional theory, the role of the cell problem is to encode the microscopic oscillation of the Hamiltonian and to produce the effective Hamiltonian through the corresponding ergodic constant.

We first record a useful chain rule in our setting.

\begin{lem} \label{lem:compose-derivative}
Let $f: V \to \R$ be Fr\'echet differentiable and let $g \in \mathcal{G}$. Then
\begin{equation*}
    D_y[f(y \circ g)] = (D_y f)(y \circ g) \circ g^{-1}, \qquad y\in V.
\end{equation*}
\end{lem}
\begin{proof}
    Define $F:V\to\R$ by $F(y):=f(y\circ g)$. Then for $h\in V$,
    \begin{align*}
        F(y+h)-F(y)&=f(y\circ g + h\circ g)-f(y\circ g)\\
        &=\langle Df(y\circ g),\, h\circ g\rangle + o(\|h\circ g\|)\\
        &=\langle (Df)(y\circ g)\circ g^{-1},\, h\rangle + o(\|h\|).
    \end{align*}
    The last equality holds because $g$ preserves the Lebesgue measure. 
\end{proof}
The next proposition constructs the ergodic constant by the standard vanishing-discount approximation and defines the effective Hamiltonian.

\begin{prop}\label{prop:effective-H}
For every \(p\in\R^d\), there exists a unique constant
\(\overline H(p)\in\R\) such that the cell problem
\begin{equation}\label{eq:cellprob}
    H(y,p\chi_I +Dv(y))=\overline H(p)
    \qquad\text{in }V
\end{equation}
admits a viscosity solution $v\in\Lip(V)$ that is
$\Lambda$-periodic and $G$-invariant. 
\end{prop}
\begin{proof}
Fix \(p\in \R^d\). For each \(\lambda>0\), we consider the discounted cell problem
\begin{equation} \label{eq:discount-cell}
    \lam v^\lambda(y)+H\bigl(y,p\chi_I + Dv^\lambda(y)\bigr)=0
    \quad\text{in } V .
\end{equation}
Note that this equation has a unique viscosity solution \(v^\lambda\). The idea is that, as \(\lambda\to0^+\), a normalized limit of \(v^\lambda\) will produce a corrector \(v\), while the limit of \(-\lambda v^\lambda\) will produce the ergodic constant. We now carry this out in several steps.

\medskip
\noindent
\textbf{Step 1. $\Lambda$-periodicity and $G$-invariance.}
Let $z\in\Lambda$, and define
\[
    w(y):=v^\lambda(y+z).
\]
By $\mathrm{(H1)}$,
\[
H\bigl(y+z,p\chi_I + Dw(y)\bigr)=H\bigl(y,p\chi_I + Dw(y)\bigr).
\]
Thus, $w$ solves the same discounted problem as $v^\lambda$. By uniqueness,
\[
v^\lambda(y+z)=v^\lambda(y)
\qquad\text{for all }y\in V \text{ and } z\in\Lambda.
\]

Next let $g\in \mathcal{G}$, and define
\[
    \widetilde w(y):=v^\lambda(y\circ g).
\]
Suppose that $\psi\in C^1(V)$ touches $\widetilde w$ from above at $y_0$. Define
$$
\phi :=\psi \circ g^{-1}.
$$
Then $\phi$ touches $v^\lambda$ from above at $y_0\circ g$, and Lemma~\ref{lem:compose-derivative} gives
$$
D\phi(y_0\circ g)=D\psi(y_0)\circ g.
$$
Using the viscosity subsolution inequality and assumption~\rm{(H2)}, we obtain
$$
\lambda\widetilde w(y_0)
+
H\bigl(y_0,p\chi_I+D\psi(y_0)\bigr)
\leq0.
$$
Note that $(p\chi_I)\circ g = p\chi_I$. The supersolution inequality follows similarly. Thus $\widetilde w$ solves \eqref{eq:discount-cell}, and uniqueness gives
$$
v^\lambda(y\circ g)=v^\lambda(y), \qquad\text{for all }y\in V \text{ and } g\in G.
$$
Thus $v^\lambda$ is both $\Lambda$-periodic and $G$-invariant.

\medskip
\noindent
\textbf{Step 2. $v^\lambda$ descends to a function on $SS^d$.}
We now show that $v^\lam$ is constant on equivalence classes of the relation defining $SS^d$. Suppose $y_1\sim y_2$. By definition, there exist sequences
$\{g_n\}\subset G$ and $\{z_n\}\subset\Lambda$ such that
\[
\|y_1-y_2\circ g_n-z_n\|_{L^2}\to 0
\qquad\text{as }n\to\infty.
\]
Using the $\Lambda$-periodicity and $G$-invariance established in Step~1,
\[
v^\lambda(y_2\circ g_n + z_n)=v^\lambda(y_2)
\qquad\text{for all }n.
\]
Since $v^\lambda$ is continuous on $V$,
\[
v^\lambda(y_1)=\lim_{n\to\infty} v^\lambda(y_2\circ g_n+z_n)=v^\lambda(y_2).
\]
Hence $v^\lam$ descends to a well-defined function on $SS^d$, still denoted by $v^\lambda ([\cdot])$.

This is precisely where the compactness of $SS^d$ becomes essential: although $V$ is infinite dimensional and noncompact, the quotient space $SS^d$ is compact, so one can apply the usual compactness tools to the family $\{v^\lambda\}_{\lambda>0}$.

\medskip
\noindent
\textbf{Step 3. Uniform bounds and passage to the limit.}
Note that \(H(\cdot,p\chi_I)\) is \(\Lambda\)-periodic and \(\mathcal G\)-invariant. Hence we may set
\[
C_0:=\max_{[y]\in SS^d} |H(y,p\chi_I)|.
\]
Then the constant functions $-C_0/\lambda$ and $C_0/\lambda$ are a viscosity subsolution and a viscosity supersolution of \eqref{eq:discount-cell}, respectively. By the comparison principle,
\[
-\frac{C_0}{\lambda}\le v^\lambda(y)\le \frac{C_0}{\lambda}
\qquad\text{for all }y\in V,
\]
and therefore
\begin{equation}\label{eq:s3_3.5}
    |\lambda v^\lambda(y)|\le C_0 \qquad\text{for all }y\in V.
\end{equation}
By coercivity, there exists $C_1>0$, independent of $\lambda$, such
that
$$
H(y,p\chi_I+q)>C_0
\qquad
\text{whenever }\|q\|>C_1.
$$
\eqref{eq:s3_3.5} implies that $v^\lambda$ is a viscosity subsolution of
$$
\|Dv^\lambda\|-C_1 \leq0.
$$
Lemma~\ref{lem:eikonal-lipschitz} therefore gives
\[
\|Dv^\lambda\|_\infty\le C_1.
\]
Since \(v^\lambda\) is Lipschitz on \(V\), for every \(x,y\in V\), $z\in \Lam$, and $g \in \mathcal G$,
\[
\begin{aligned}
|v^\lambda(x)-v^\lambda(y)| =
|v^\lambda(x)-v^\lambda(y\circ g+z)| \le C_1\|x-y\circ g-z\|.
\end{aligned}
\]
Taking the infimum gives
\[
    |v^\lambda(x)-v^\lambda(y)|
    \le C_1 \cdot d_{SS^d}([x],[y]).
\]
Thus \(v^\lambda\) descends to a Lipschitz continuous function on \(SS^d\).

Fix a reference point $\bar y\in V$, and define
\[
w^\lambda([y]):=v^\lambda([y])-v^\lambda([\bar y]), \qquad y \in V.
\]
Then $\{w^\lambda \}_{\lambda>0}$ is equi-Lipschitz on $SS^d$. Also, it is equibounded since $SS^d$ has finite diameter. By Arzel\`a-Ascoli, there exist a sequence $\lambda_j\to 0$, a constant $c\in\R$, and a function $w\in \Lip(SS^d)$ such that
\[
    w^{\lambda_j}\to w \qquad\text{uniformly on }SS^d,
\]
and
\[
    -\lambda_j v^{\lambda_j}([\bar y])\to c.
\]
as $j \to \infty$. Since $\widetilde w^\lam :=  w^\lambda \circ \pi = v^\lam - v^\lam (\bar y)$ solves
\[
\lambda \widetilde w^\lambda(y) + H\bigl(y,p\chi_I + D\widetilde w^\lambda(y)\bigr)
= -\lambda v^\lambda(\bar y)
\qquad\text{in }V,
\]
the stability of viscosity solutions implies that $v:= w \circ \pi$ solves
\[
H\bigl(y,p\chi_I + Dv(y)\bigr)=c
\qquad\text{in }V.
\]

\medskip
\noindent
\textbf{Step 4. Uniqueness of the ergodic constant.}
Suppose $(v_1,c_1)$ and $(v_2,c_2)$ are two correctors and corresponding ergodic constants of the cell problem with $c_1<c_2$. Since $v_1,v_2$ are bounded on $SS^d$, we may choose $\delta>0$ small enough so that
\[
\delta v_1([y]) + c_1 < \delta v_2([y])+c_2
\qquad\text{for all }y\in V.
\]
Equivalently,
\[
\delta v_1(y)+H\bigl(y,p\chi_I + Dv_1(y)\bigr)
< \frac{c_1+c_2}{2}<
\delta v_2(y)+H\bigl(y,p\chi_I + Dv_2(y)\bigr)
\qquad\text{in }V.
\]
By the comparison principle for the discounted equation, we obtain
\[
v_1\le v_2 \qquad\text{on }V.
\]
For any fixed \(C>0\), choosing \(\delta>0\) even smaller if necessary, the same comparison applied to \(v_1+C\) gives \(v_1+C\le v_2\), which is impossible for large \(C\). Therefore $c_1=c_2$, and the ergodic
constant is unique.

We thus define
\[
\overline H(p):=c,
\qquad p\in\mathbb R^d.
\]
As a consequence of the uniqueness of the ergodic constant, the
vanishing-discount convergence holds for the full family:
\[
-\lambda v^\lambda(\cdot;p)
 \longrightarrow\overline H(p)
\qquad\text{uniformly on }V
\quad\text{as }\lambda\downarrow0.
\]
This completes the proof.
\end{proof}

Throughout the rest of the paper, when a function \(v\) is defined on \(SS^d\), we identify it with its lift \(v\circ\pi\) to \(V\), where $\pi$ denotes the canonical quotient map. Thus the cell problem is always understood in the viscosity sense for the lifted function on \(V\). This convention allows us to use the compactness of \(SS^d\) while keeping the viscosity structure on the Hilbert space \(V\).

Proposition~\ref{prop:effective-H} gives the effective Hamiltonian on $\R^d\cong Y^\perp$, the space of momenta relevant to the effective dynamics. We next record its basic properties.

\begin{prop}\label{prop:3.3}
The effective Hamiltonian
\[
\overline H:\mathbb R^d\to\mathbb R
\]
is coercive and locally Lipschitz. More precisely, for every \(R>0\)
there exists \(C_R>0\) such that
\begin{equation}\label{s3_5}
|\overline H(p)-\overline H(q)|
\leq C_R|p-q|
\qquad
\text{for all }p,q\in B_{\R^d}(0,R).
\end{equation}
\end{prop}

\begin{proof}
Let \(v_p\) be a corrector for \(p\chi_I\). At a maximum point $y_{\max}$ of \(v_p\), the constant test function gives, by the viscosity subsolution inequality,
\[
H(y_{\max},p\chi_I)\leq\overline H(p).
\]
At a minimum point $y_{\min} \in V$ of $v_p$, the viscosity supersolution inequality gives
\[
\overline H(p)\leq H(y_{\min},p\chi_I).
\]
Therefore,
\begin{equation*}
\min_{[y]\in SS^d}H(y,p\chi_I)
\leq
\overline H(p)
\leq
\max_{[y]\in SS^d}H(y,p\chi_I).
\end{equation*}
Coercivity of \(H\) implies
\[
\overline H(p)\to+\infty
\qquad\text{as} \qquad |p|\to\infty.
\]

We next prove local Lipschitz continuity. Let $v^\lam(\cdot;p)$ be the viscosity solution of the discounted equation~\eqref{eq:discount-cell}. Fix \(R>0\). The proof of Proposition~\ref{prop:effective-H} gives a constant \(C_R>0\) such that
\[
\operatorname{Lip}
v^\lambda(\,\cdot\,;p)\leq C_R
\]
for every \(p\in B(0,R)\) and \(\lambda\in(0,1]\). Let \(L_{R}\) be the momentum Lipschitz constant from \rm{(H3)} on $B_V(0,R+C_R)$. If \(p,q\in B(0,R)\), comparison of the two discounted equations gives
\[
v^\lambda(\,\cdot\,;q)
-\frac{L_R}{\lambda}|p-q|
\leq
v^\lambda(\,\cdot\,;p)
\leq
v^\lambda(\,\cdot\,;q)
+\frac{L_R}{\lambda}|p-q|.
\]
Thus
\[
\lambda
\left|
v^\lambda(y;p)-v^\lambda(y;q)
\right|
\leq
L_R|p-q|.
\]
Letting \(\lambda\downarrow0\), the uniqueness of the ergodic constant gives
\[
    |\overline H(p)-\overline H(q)| \leq L_R|p-q|.
\]
This proves the claim.
\end{proof}

\subsection{The effective equation} 
We are now ready to identify the equation satisfied by the locally uniform limit of $u^\e $. 

\begin{proof}[Proof of Theorem~1.1]
Let \(u\) be a locally uniform limit along a subsequence \(\varepsilon\to0\), and write
\[
u(x,t)=\widetilde u(\mathfrak m(x),t),
\qquad x\in V,\quad t\in[0,\infty).
\]
We first establish the subsolution inequality for \(u\) and then express it in terms of \(\widetilde u\). The supersolution property follows analogously.

Let $\phi \in C^1(V\times [0,\infty))$ touch $u$ from above at $(x_0,t_0)$. Then,
\[
    D_x\phi(x_0,t_0)=p_0\chi_I
\]
for some $p_0 \in \R^d$. Let $v\in \Lip(V)$ be a solution of the cell problem associated with $p_0\chi_I$, namely
\[
    H(y,p_0\chi_I + Dv(y))=\overline H(p_0) \qquad\text{in }V.
\]
Fix \(r>0\) small enough so that \(t_0-r>0\), and set
\[
    Q_r:=\overline{B(x_0,r)}\times [t_0-r,t_0+r].
\]
By replacing \(\phi\) with $\phi+\|x-x_0\|^2+|t-t_0|^2 $, if necessary, we may assume that \(u-\phi\) has a strict maximum at $(x_0,t_0)$ in $Q_r$, and that
\begin{equation}\label{eq:strict-local-touch}
    u(x,t)-\phi(x,t) \le u(x_0,t_0)-\phi(x_0,t_0) -\|x-x_0\|^2-|t-t_0|^2 \qquad \text{in} \;Q_r .
\end{equation}
In particular, 
\begin{equation}\label{eq:outside-penalty}
    u(x,t)-\phi(x,t) \le u(x_0,t_0)-\phi(x_0,t_0)-r^2 \qquad \text{on } \partial Q_r .
\end{equation}

\medskip
\noindent
\textbf{Step 1. Doubling the variables.}
For $\eta>0$, define
\[
\Phi(x,y,t)
:=
u^\e (x,t)-\phi(x,t)-\e  v(y)-\frac1\eta\left\|y-\frac{x}{\e }\right\|^2,
\qquad (x,t)\in Q_r,\ y\in V .
\]
Choose a sequence $(x_n,t_n, y_n) \in Q_r\times V$ such that
\[
\Phi(x_n,y_n,t_n)\;\uparrow\; \sup_{Q_r\times V} \Phi
\]
and
\[
\Phi(x_n,y_n,t_n)\ge \Phi\!\left(x_n,\frac{x_n}{\e },t_n\right).
\]
From the inequality above, we immediately obtain
\begin{equation}\label{eq:xy-close}
\left\|y_n-\frac{x_n}{\e }\right\|^2 \le C\e \eta.
\end{equation}
\eqref{eq:xy-close} implies that the fast variables \(y_n\) and \(x_n/\e \) remain \(O(\sqrt{\e \eta})\)-close, which is the estimate needed later when comparing the two Hamiltonians. 

Since \(v\) is bounded, we have
\[
\Phi(x_0,x_0/\e ,t_0)
\ge u(x_0,t_0)-\phi(x_0,t_0)-\omega_r(\e )
-\e \|v\|_\infty ,
\]
where \(\omega_r(\e )\to0\) is the modulus of local uniform convergence
of \(u^\e \to u\) on \(Q_r\). 

On the other hand, if \((x,t)\in \partial Q_r\), then
\[
\Phi(x,y,t) \le u(x_0,t_0)-\phi(x_0,t_0)
-r^2+\omega_r(\e )+\e \|v\|_\infty .
\]
Therefore, for \(\e >0\) sufficiently small,
\begin{equation}\label{eq:boundary-bound}
\sup_{\partial Q_r\times V}\Phi < \Phi(x_0,x_0/\e ,t_0) - \frac{r^2}{2}.
\end{equation}
Indeed, \eqref{eq:boundary-bound} implies that the maximizing sequence \((x_n,t_n)\) stays away from $\partial Q_r$.

\medskip
\noindent
\textbf{Step 2. Find the maximizer using the Radon--Nikodym property.}
Fix $\gamma>0$, and localize near $(x_n,y_n,t_n)$ by setting
\[
    S_n :=
    \left\{
    (x,y,t)\in V\times V\times (0,\infty):
    \left\|\frac{x}{\e }-\frac{x_n}{\e }\right\|^2 + \|y-y_n\|^2 + |t-t_n|^2 \le \gamma^2
    \right\}.
\]
Choose \(\gamma>0\) small enough so that the projection of \(S_n\) onto the \((x,t)\)-variables is contained in $Q_r$. Let
\[
    \delta_n := \sup_{Q_r\times V} \Phi - \Phi(x_n,y_n,t_n),
\]
and define
\[
    \Psi_n(x,y,t) :=
    \Phi(x,y,t)
    -\frac{2\delta_n}{\gamma^2}
    \left(
    \left\|\frac{x}{\e }-\frac{x_n}{\e }\right\|^2 + \|y-y_n\|^2 + |t-t_n|^2
    \right)
    \qquad\text{on }S_n.
\]
If $p: S_n \rightarrow\R$ varies less than $\del_n$ over $S_n$ and $\Psi_n + p$ attains its maximum on $S_n$, then this point must be interior to $S_n$. By the Radon--Nikodym property, there exist $p_n,q_n\in V$ and $k_n\in \R$ with
\begin{equation}\label{eq:perturb-small}
\|p_n\| + \|q_n\| +  |k_n| \le \frac{\delta_n}{2\gamma},
\end{equation}
such that
\[
(x,y,t)\mapsto \Psi_n(x,y,t)-\langle p_n,x/\e\rangle-\langle q_n,y\rangle-k_n t
\]
attains its maximum over $S_n$ at some point $(\hat x_n,\hat y_n,\hat t_n)\in \text{int}(S_n)$. In particular,
\begin{equation}\label{eq:yx-close2}
    \left\|\hat y_n-\frac{\hat x_n}{\e }\right\| \le 2\gamma + C\sqrt{\e \eta}.
\end{equation}

\medskip
\noindent
\textbf{Step 3. Subsolution and supersolution tests.}
Since $(\hat x_n,\hat t_n)$ is a local maximum of
\[
    (x,t)\mapsto
    \Psi_n(x,\hat y_n,t)-\langle p_n,x/\e\rangle-\langle q_n,\hat y_n\rangle-k_n t,
\]
the viscosity subsolution test for $u^\e $ yields
\begin{multline}\label{eq:1-subsol-ineq}
\phi_t(\hat x_n,\hat t_n)
+\frac{4\delta_n}{\gamma^2}(\hat t_n-t_n)
+k_n
+
H\!\left(
\frac{\hat x_n}{\e },
D\phi(\hat x_n,\hat t_n)
+\frac{2}{\eta\e }\left(\frac{\hat x_n}{\e }-\hat y_n\right)
+\mathcal{E}_n^1
\right)
\le 0,
\end{multline}
where
\[
\mathcal{E}_n^1
=
\frac{4\delta_n}{\gamma^2\e }\left(\frac{\hat x_n}{\e }-\frac{x_n}{\e }\right)+ \frac{p_n}{\e}.
\]

Likewise, since $\hat y_n$ is a local minimum of
\[
y\mapsto
\e  v(y)
+
\frac1\eta\left\|y-\frac{\hat x_n}{\e }\right\|^2
+
\frac{2\delta_n}{\gamma^2}\|y-y_n\|^2
+\langle q_n,y\rangle,
\]
the viscosity supersolution test for the cell problem gives
\begin{equation}\label{eq:1-supersol-ineq}
-\overline H(p_0)
+
H\!\left(
\hat y_n,
p_0\chi_I
+\frac{2}{\eta\e }\left(\frac{\hat x_n}{\e }-\hat y_n\right)
+\mathcal{E}_n^2
\right)
\ge 0,
\end{equation}
where
\[
\mathcal{E}_n^2 = -\frac{4\delta_n}{\gamma^2\e}(\hat y_n-y_n)- \frac{q_n}{\e}.
\]

By construction,
\begin{equation}\label{eq:E-small}
\|\mathcal{E}_n^1\|\to 0,
\qquad
\|\mathcal{E}_n^2\|\to 0
\qquad\text{as }n\to\infty
\end{equation}
because $\del_n \to 0$ as $n\to \infty$.

\medskip
\noindent
\textbf{Step 4. Convergence of $(\hat{x}_n, \hat{t}_n)$.}
Since $\Phi(\hat x_n,\hat y_n,\hat t_n)$ converges to $\sup_{Q_r\times V} \Phi$ as $n \to \infty$ and $\gam \to 0$, we may compare with $\Phi(x,x/\e ,t)$ and obtain
\begin{multline*}
u^\e (x,t)-\phi(x,t)-\e  v\!\left(\frac{x}{\e }\right)
\le \\
\lim_{\gam \to 0}\limsup_{n\to\infty}
\left[
u^\e (\hat x_n,\hat t_n)-\phi(\hat x_n,\hat t_n)-\e  v(\hat y_n)
-\frac1\eta\left\|\hat y_n-\frac{\hat x_n}{\e }\right\|^2
\right].
\end{multline*}
Letting $\e \to 0$, and using
\eqref{eq:yx-close2}, we arrive at
\[
u(x,t)-\phi(x,t)
\le
\lim_{\e \to 0^+}\lim_{\gam \to 0}\limsup_{n\to\infty}
\bigl(u^\e (\hat x_n,\hat t_n)-\phi(\hat x_n,\hat t_n)\bigr).
\]
From the uniform convergence of $u^\e \to u$ on $Q_r$ and the condition \eqref{eq:strict-local-touch}, we conclude that 
\begin{equation*}
    \lim_{\e \to 0}\lim_{\gam\to0}\limsup_{n\to\infty}
\left(
\|\hat{x}_n-x_0\|^2+|\hat t_n-t_0|^2
\right)=0.
\end{equation*}

This implies that
\[
(\hat x_n,\hat t_n)\to (x_0,t_0)
\]
up to subsequences, as $n\to\infty, \gam\to 0, \e \to 0$.

\medskip
\noindent

\textbf{Step 5. Passage to the limit.}
Define
\[
P_n
:=
D\phi(\hat x_n,\hat t_n)
+
\frac{2}{\eta\varepsilon}
\left(
\frac{\hat x_n}{\varepsilon}-\hat y_n
\right)
+\mathcal E_n^1
\]
and
\[
Q_n
:=
p_0 \chi_I
+
\frac{2}{\eta\varepsilon}
\left(
\frac{\hat x_n}{\varepsilon}-\hat y_n
\right)
+\mathcal E_n^2.
\]
\eqref{eq:1-subsol-ineq} implies
\[
H\left(\frac{\hat x_n}{\varepsilon},P_n\right)
\leq
-\phi_t(\hat x_n,\hat t_n)
-\frac{4\delta_n}{\gamma^2}(\hat t_n-t_n)
-k_n.
\]
For fixed \(\varepsilon,\eta,\gamma>0\), the right-hand side is bounded uniformly for all sufficiently large \(n\). Coercivity therefore gives
\[
\|P_n\|\leq R
\]
for all sufficiently large $n$. The threshold index may depend on the auxiliary parameters, but $R$ does not.

Since \(D\phi(\hat x_n,\hat t_n)\) and \(\mathcal E_n^1\) are bounded for fixed $\e, \eta$, and $\gam$,
it follows that
\[
\frac{2}{\eta\varepsilon}
\left(
\frac{\hat x_n}{\varepsilon}-\hat y_n
\right)
\]
is uniformly bounded for sufficiently large $n$. Consequently, \(Q_n\) is uniformly bounded as well. Thus one fixed local momentum Lipschitz constant from \rm{(H3)} may be used in both viscosity inequalities.

Subtracting \eqref{eq:1-supersol-ineq} from \eqref{eq:1-subsol-ineq}, and using the local Lipschitz regularity of $H$, we obtain
\begin{multline} \label{eq:1-final-ineq}
    \phi_t(\hat{x}_n, \hat{t}_n) + \overline{H}(p_0) \leq -k_n - \frac{4\del_n}{\gam^2}(\hat{t}_n - t_n) \\
    + C(\|\hat{y}_n - \frac{\hat{x}_n}{\e }\|+\|p_0 \chi_I - D\phi(\hat{x}_n, \hat{t}_n)\| + \|\mathcal{E}_n^1\| + \|\mathcal{E}_n^2\|).
\end{multline}

The limits are taken in the following order: first \(n\to\infty\) with \(\varepsilon,\eta,\gamma\) fixed; then \(\gamma\downarrow0\); and finally \(\varepsilon\downarrow0\), with \(\eta>0\) fixed. Using \eqref{eq:perturb-small}, \eqref{eq:yx-close2}, \eqref{eq:E-small}, \eqref{eq:1-final-ineq} and the convergence $(\hat x_n,\hat t_n)\to (x_0,t_0)$, we deduce
\begin{equation} \label{eq:s3_15}
\phi_t(x_0,t_0)+\overline H(p_0)\le 0.
\end{equation}
This implies that $u$ is a viscosity subsolution of the effective equation.

\medskip
\noindent
\textbf{Step 6. Finite-dimensional limit.}
We now express the preceding viscosity inequality in terms of the finite-dimensional limit. Define $\widetilde u: \R^d \times [0,\infty)$ by
$$
    u(x,t) = \widetilde u(\mathfrak m(x),t),
    \qquad
    (x,t)\in\mathbb V\times[0,\infty).
$$
Let $\psi \in C^1(\R^d\times [0,\infty))$ touch $\widetilde u$ from above at $(q_0,t_0)$. Then 
\begin{equation*}
    \phi(x,t) := \psi(\mathfrak m(x), t)\in C^1(V \times [0,\infty)), \qquad (x,t) \in V\times [0,\infty)
\end{equation*}
touches $u$ from above at $(q_0\chi_I, t_0)$ and 
\begin{equation*}
    \phi_t(q_0\chi_I,t_0) = \psi_t(q_0,t_0),\qquad D\phi(q_0\chi_I, t_0) = D\psi(q_0,t_0) \chi_I.
\end{equation*}
\eqref{eq:s3_15} implies
\[
    \psi_t(q_0,t_0) + \overline H (D\psi(q_0,t_0)) \leq 0.
\]
Moreover, for every $q\in\mathbb R^d$,
\[
\widetilde u(q,0)
=u(q\chi_I,0)
=u_0(q\chi_I)
=\widetilde u_0(q).
\]
Thus, $\widetilde u$ is a viscosity subsolution of $\mathrm{(\overline{CP})}$.

The supersolution property is proved in the same way by testing with a function touching $u$ from below. Therefore, $\widetilde u$ is the viscosity solution of $\mathrm{(\overline {CP})}$. Proposition~\ref{prop:3.3} and the standard finite-dimensional comparison principle imply that this solution is unique. Hence the whole family $u^\e $ converges locally uniformly to $u$.
\end{proof}

This completes the qualitative homogenization argument. The cell problem has produced the effective Hamiltonian, and the perturbed test function method has identified the effective equation satisfied by the macroscopic limit. In the next section, we refine this argument to obtain the quantitative convergence rate. 

\begin{rem}
As already emphasized in Remark~\ref{rem1}, the effective Hamiltonian is only relevant on the finite-dimensional space $Y^\perp \cong \R^d$. Likewise, one may associate the original Hamiltonian $H$ with the finite-dimensional Hamiltonian
\begin{equation*}
    \widetilde H(\tilde x,\tilde p):=
    H(\tilde x\chi_I,\tilde p\chi_I),
    \qquad
    (\tilde x,\tilde p)\in \R^d\times \R^d,
\end{equation*}
and consider the corresponding homogenization problem
\[
\begin{cases}
\widetilde u_t^\e + \widetilde H\left(\dfrac{\tilde x}{\e},D\widetilde u^\e\right)=0
& \text{in }\R^d\times(0,\infty),\\
\widetilde u^\e(\tilde x,0)=\widetilde u_0(\tilde x)
& \text{on }\R^d.
\end{cases}
\]
By the standard finite-dimensional homogenization theory, this problem converges to a certain effective equation. However, this effective equation need not coincide with $\mathrm{(\overline{CP})}$, because the effective Hamiltonian associated with \(\widetilde H\) need not equal \(\overline H\). More precisely, the cell problem defining \(\overline H\) is posed on the full quotient \(SS^d\), not only on the constant configurations. Therefore, any comparison between the solution of the induced finite-dimensional oscillatory problem, $\widetilde u^\e (\tilde x, t)$, and the restriction \(u^\e (\tilde x\chi_I,t)\) requires an additional argument or assumption.
\end{rem}


\section{The rate of convergence in homogenization}\label{sec4} 
In this section, we prove Theorem~\ref{thm2}. We first make a standard reduction of the Hamiltonian that gives global control in the momentum variable.

\begin{lem}\label{lem4.1}
Under the assumptions of Theorem~\ref{thm2}, we may assume without loss of generality that there exists \(L>0\) such that
\begin{equation}\label{eq:s4_1}
    |H(x,p)-H(y,q)| \leq L\bigl(\|x-y\|+\|p-q\|\bigr)
\end{equation}
for all \(x,y,p,q\in V\), and that
\begin{equation}\label{eq:s4_2}
    L^{-1}\|p\|-L \leq H(x,p) \leq L(1+\|p\|)
\end{equation}
for all \(x,p\in V\).
\end{lem}

\begin{proof}
By Proposition~\ref{prop:uniform-lipschitz}, we have
\begin{equation}
    \|Du^\e\|_\infty \le L_1
\end{equation}
for every $\e > 0$.  For \(p\in B_{\mathbb R^d}(0,L_1+1)\), let $v_p$ be the solution of the cell problem with $p\chi_I$. In the proof of Proposition~\ref{prop:effective-H}, we have
\[
    \|Dv_p\|_\infty \leq L_2.
\]
Choose $R = L_1 + L_2 + 2$. Define \(\Pi_R:V\to\overline {B_V(0,R)}\) by 
\[
\Pi_R(p) =
\begin{cases}
p,&\|p\|\leq R,\\[1mm]
R\,p/\|p\|,&\|p\|>R.
\end{cases}
\]
Define
\begin{equation}\label{eq:s4_3}
H_R(x,p)
:=
H(x,\Pi_Rp)+(\|p\|-R)_+.
\end{equation}
Since \(\|\Pi_Rp\| \le R\), Lemma~\ref{lem:H-bounded-on-bounded-p} gives
\[
C_R := \sup_{x,p \in V}|H(x,\Pi_Rp)|<\infty.
\]
Consequently, \(H_R\) satisfies the two-sided linear-growth estimate~\eqref{eq:s4_2} by choosing
\[
    L = 1+R+C_R.
\]
The nonexpansiveness of \(\Pi_R\), together with \rm{(H3)}, gives the global Lipschitz estimate~\eqref{eq:s4_1}. Moreover, \(H_R\) still satisfies \rm{(H1)} and \rm{(H2)}.

By construction, all viscosity tests of \(u^\varepsilon\) involve momenta in \(B_V(0,L_1+1)\), where \(H_R=H\). Thus the oscillatory solutions are unchanged. Similarly, for \(p\in B_{\mathbb R^d}(0,L_1+1)\), all momenta occurring in viscosity tests of the corrector \(v_p\) belong to \(B_V(0,R)\). Hence \(v_p\) solves the cell problem for both \(H\) and \(H_R\), and the corresponding effective Hamiltonians coincide on \(B_{\mathbb R^d}(0,L_1+1)\). Thus, the effective equation solved by \(u\) is unchanged.

We may therefore replace \(H\) by \(H_R\). For notational simplicity, we continue to denote the modified Hamiltonian and its effective Hamiltonian by \(H\) and \(\overline H\), respectively.
\end{proof}

We next record the estimates for the discounted correctors that will be used in the proof of Theorem~\ref{thm2}. For $p\in \R^d$, let $v^\lambda(\cdot;p)$ denote the viscosity solution of the discounted cell problem
\[
\lambda v^\lambda(y;p) + H\bigl(y,p\chi_I + Dv^\lambda(y;p)\bigr)=0
\qquad\text{in }V.
\]

\begin{lem}\label{lem4.2}
There exists $C>0$, independent of $\lambda$, such that, for all $x,y\in V$ and $p,q\in\mathbb R^d$,
\begin{equation}\label{eq:vlambda-lip-p-sec4}
    |v^\lambda(x;p) - v^\lam(y;p)| \le C(1+|p|)\|x-y\|, \qquad\lambda |v^\lambda(y;p)-v^\lambda(y;q)|
    \le C|p-q|
\end{equation}
and
\begin{equation}\label{eq:vlambda-ergodic-sec4}
    |\lambda v^\lambda(y;p)+\overline H(p)|
    \le C\lambda(1+|p|).
\end{equation}
\end{lem}
\begin{proof} 
By~\eqref{eq:s4_2}, $|H(y,p\chi_I)|\le L(1+|p|)$. Comparison gives
\[
|\lambda v^\lambda(\cdot;p)|\le L(1+|p|).
\]
This implies that $v^\lambda(\cdot;p)$ is a viscosity subsolution of $\|Dv^\lambda\|-C(1+|p|)\le0$. Lemma~\ref{lem:eikonal-lipschitz} therefore gives the spatial Lipschitz estimate in~\eqref{eq:vlambda-lip-p-sec4}.

Next, using~\eqref{eq:s4_1}, together with comparison with
\[
v^\lambda(\cdot;q)\pm\frac{L}{\lambda}|p-q|,
\]
gives the momentum Lipschitz estimate in~\eqref{eq:vlambda-lip-p-sec4}.

Finally, to prove~\eqref{eq:vlambda-ergodic-sec4}, let \(v_p\) be a normalized corrector solving the cell problem for \(p\chi_I\). The preceding argument gives
\[
\operatorname{osc}_{V}v_p\le C(1+|p|).
\]Comparing \(v^\lambda(\cdot;p)\) with
\[
    v_p-\frac{\overline H(p)}{\lambda}\pm \operatorname{osc}_{V}v_p
\]
gives the result. 
\end{proof}

We record the following finite-dimensional lemma. See \cite{capuzzo-rate} for details.
\begin{lem}\label{lem4.3}
Let \(f,g:\R ^m\to\R \) be
locally Lipschitz. If
\[
    0\in \overline D^-(f+g)(z),
\]
then there exists \(r\in\R ^m\) such that
\[
    r\in \overline D^-f(z),
    \qquad
    -r\in \overline D^-g(z).
\]
\end{lem}
We use this lemma in the finite-dimensional variables \((y,s)\in Y^\perp\times \R \cong\R ^{d+1}\). The following lemma allows us to apply the viscosity inequality when a contact point occurs at the terminal time $t=T$.

\begin{lem}
\label{lem:terminal-contact}
Let $w$ be a viscosity subsolution of
\[
w_t+F(x,Dw)=0\qquad\text{in }V\times(0,T+\rho)
\]
for some $\rho>0$.  Suppose that $\phi\in C^1(V\times(0,T])$ and that $w-\phi$ has a strict maximum at $(x_0,T)$ relative to $V\times(0,T]$.  Then
\[
\phi_t(x_0,T)+F(x_0,D\phi(x_0,T))\le0.
\]
The analogous assertion holds for supersolutions touched from below.
\end{lem}

\begin{proof} We may assume that, near \((x_0,T)\),
\begin{equation*}
    w(x,t) - \phi(x,t)\leq w(x_0,T) - \phi(x_0, T) - \|x - x_0\|^2 - |t - T|^2.
\end{equation*}
Choose $0<r<T$ and $\sigma>0$ such that 
$$
    w(x,t) - \phi(x,t)\leq w(x_0,T) - \phi(x_0, T)-4\sigma
$$
on the lateral and lower boundary of
$$
    C:=\overline {B(x_0,r)} \times[T-r,T].
$$

For $\eta>0$, define
$$
    \Phi_\eta(x,t) := w(x,t) - \phi(x,t) -\frac{\eta}{T-t}, \qquad t<T.
$$
Let
$$
\rho_\eta:=\eta^2, \qquad z_\eta:=(x_0,T-\sqrt{\eta})
$$
and consider
$$
C_\eta := \overline {B(x_0,r)} \times[T-r,T-\rho_\eta].
$$
For sufficiently small $\eta$, we have $z_\eta\in\operatorname{int}C_\eta$, and
$$
    \Phi_\eta(z_\eta) \rightarrow w(x_0,T) - \phi(x_0,T) \qquad \text{as} \qquad \eta \to 0.
$$
Hence,
$$
    \Phi_\eta(z_\eta)\geq w(x_0,T) - \phi(x_0,T)-\sigma
$$
for all sufficiently small $\eta$.

On the lateral and lower boundary of $C_\eta$, 
$$
\Phi_\eta(x,t)\leq w(x_0,T) - \phi(x_0, T)-4\sigma.
$$
On the upper boundary $t=T-\rho_\eta$, we have
$$
\Phi_\eta(x,T-\rho_\eta)
= w(x,T-\rho_\eta) - \phi(x,T - \rho_\eta)-\frac{1}{\eta}
\leq
w(x_0,T) - \phi(x_0,T) -4\sigma
$$
for sufficiently small $\eta$. Therefore, for sufficiently small $\eta$,
$$
\sup_{\partial C_\eta} \Phi_\eta
\leq
w(x_0,T)-\phi(x_0,T)-4\sigma.
$$

By the Radon--Nikodym property, there exist $p_\eta\in V$ and $k_\eta\in\mathbb R$ such that
$$
\|p_\eta\|+|k_\eta|\to0 \qquad \text{as}\qquad \eta \to 0
$$
and
$$
(x,t)\longmapsto
\Phi_\eta(x,t)
-\langle p_\eta,x-x_0\rangle
-k_\eta(t-T)
$$
attains its maximum on $C_\eta$ at some point $(x_\eta,t_\eta)$. Choose the perturbation sufficiently small that its oscillation on $C_\eta$ is less than $\sigma$. Then $(x_\eta,t_\eta)\in\operatorname{int}C_\eta$.

By the maximality of $(x_\eta,t_\eta)$, comparing with $z_\eta$ gives
$$
w(x_\eta,t_\eta) - \phi(x_\eta,t_\eta)
\geq
w(z_\eta)-\phi(z_\eta)-\sqrt{\eta}-o_\eta(1),
$$
where the error tends to zero as $\eta\to0$. Since $z_\eta\to(x_0,T)$, it follows that
$$
\liminf_{\eta\downarrow0}w(x_\eta,t_\eta) - \phi(x_\eta,t_\eta)
\geq w(x_0,T) - \phi(x_0,T).
$$
On the other hand, we have
$$
w(x_\eta,t_\eta) - \phi(x_\eta,t_\eta)
\leq w(x_0,T) - \phi(x_0,T)
-\|x_\eta-x_0\|^2 -|t_\eta-T|^2.
$$
Consequently, $(x_\eta,t_\eta) \to (x_0,T)$ as $\eta \to 0$.

The viscosity subsolution inequality gives
$$
    \phi_t(x_\eta,t_\eta) + \frac{\eta}{(T-t_\eta)^2}
    + k_\eta + F\bigl(x_\eta,D\phi(x_\eta,t_\eta)+p_\eta\bigr) \leq0.
$$
Dropping the nonnegative singular term and letting $\eta\downarrow0$, we obtain
$$
\phi_t(x_0,T)+F\bigl(x_0,D\phi(x_0,T)\bigr)\leq0.
$$
The supersolution statement follows by reversing the signs.
\end{proof}
We now prove the quantitative convergence estimate.
\begin{proof}[Proof of Theorem~1.2.]
    We prove the upper bound
    \begin{equation*}
        \sup_{V \times [0,T]} u^\e (x,t) - \widetilde u(\mathfrak m(x),t) \le C_T\e ^{1/3}.
    \end{equation*}
    The lower bound is obtained by a symmetric argument. Throughout the proof, \(C>0\) denotes a constant that may change from line to line and may depend on \(H\), \(u_0\), \(d\), and the fixed time horizon \(T\), but is independent of $\e$ and other parameters in the proof.
    
    \medskip
    \noindent
    \textbf{Step 1. Doubling the variables.}
    Fix parameters
    \[
            \beta,\theta, \del\in(0,1),\qquad \lambda=\e ^\theta , \qquad K > 0
    \]
    with $\theta+\beta<1$. The particular choice will be made at the end.
    
    For \(x\in V\), \(y\in \R^d\), and \(t,s\in[0,T]\), define
    \begin{align*}
    \Phi(x,y,t,s) &:=u^\e (x,t)- \widetilde u(y,s)
            -\e  v^\lambda\left(\frac{x}{\e };
            \frac{\mathfrak m(x)-y}{\e ^\beta}\right)      \\
    &\quad
            -\frac{|\mathfrak m(x)-y|^2+|t-s|^2}{2\e ^\beta}
            -K(t+s)-\delta|y|^2.
    \end{align*}
    Also, we may restrict \(x\) to representatives satisfying
    \[
            \|x-Mx\|\le C_1\e
    \]
    where $C_1$ is the constant in Lemma~\ref{lem:epsilon-decomposition}. By Proposition~\ref{prop:compactness} and $\Lam$-periodicity of $v^\lam$, this restriction does not change the supremum of \(\Phi\). 
    
    The linear growth estimate for $\overline H$, together with Lemma~\ref{lem4.2}, implies
    \[
    -\varepsilon v^\lambda\left(
    \frac{x}{\varepsilon};P \right)
    \leq
    C\varepsilon^{1-\theta}(1+|P|), \qquad P=\frac{\mathfrak m(x)-y}{\varepsilon^\beta}.
    \]
    Consequently,
    \begin{align*}
        \Phi(x,y,t,s) \leq
        C+C\varepsilon^{1-\theta}
        +C\varepsilon^{1-\theta-\beta}| \mathfrak m(x)-y|
        -\frac{|\mathfrak m(x)-y|^2}{2\varepsilon^\beta}
        -\delta|y|^2.
    \end{align*}
    Because \(\theta+\beta<1\), the negative quadratic term dominates the linear term in \(|\mathfrak m(x)-y|\). Thus \(\Phi\) is bounded above.
    
    Let \((x_n,y_n,t_n,s_n)\) be a maximizing sequence for \(\Phi\). Namely,
    \[
        \Phi(x_n,y_n,t_n,s_n)\to\sup_{V \times \R^d \times [0,T]^2}\Phi, \qquad \|x_n-Mx_n\|\le C_1\e .
    \]

    Comparison with \(\Phi(0,0,0,0)\), together with Lemma~\ref{lem4.2}, gives
    \begin{equation*}
        \delta|y_n|^2 +
        \frac{|\mathfrak m(x_n)-y_n|^2+|t_n-s_n|^2}{2\varepsilon^\beta}
        \leq
        C\bigl(1+\varepsilon^{1-\theta}\bigr)
        +C\varepsilon^{1-\theta-\beta} |\mathfrak m(x_n)-y_n|+o_n(1).
    \end{equation*}
    Young's inequality yields
    \begin{equation}\label{eq:s4_7}
    |y_n|\leq C\delta^{-1/2}+o_n(1).
    \end{equation}
    
    Next, comparing \(\Phi(x_n,y_n,t_n,s_n)\) with \(\Phi(x_n,\mathfrak m(x_n),t_n,t_n)\), we obtain
    \begin{align*}
        \frac{|\mathfrak m(x_n)-y_n|^2+|t_n-s_n|^2}
             {2\varepsilon^\beta}
        &\leq
        C|\mathfrak m(x_n)-y_n|+C|t_n-s_n|
        +C\varepsilon^{1-\theta-\beta}|\mathfrak m(x_n)-y_n|
        \\
        &\quad
        +\delta\bigl(|\mathfrak m(x_n)|^2-|y_n|^2\bigr)+o_n(1).
    \end{align*} 
    Since $\delta\bigl(|\mathfrak m(x_n)|^2-|y_n|^2\bigr)
        \leq
        \delta|\mathfrak m(x_n)-y_n|^2
        +C\delta^{1/2}|\mathfrak m(x_n)-y_n|$,
    \begin{equation}\label{eq:s4_8}
        |\mathfrak m(x_n)-y_n|+|t_n-s_n|
        \leq
        C\varepsilon^\beta+o_n(1).
    \end{equation}
    In particular,
    \[
        P_n:=\frac{\mathfrak m(x_n)-y_n}{\varepsilon^\beta}
    \]
    is uniformly bounded. 
    By~\eqref{eq:s4_7}--\eqref{eq:s4_8},
    \[
        \|x_n\| \leq
        \|x_n-Mx_n\|+|\mathfrak m(x_n)-y_n|+|y_n| \leq
        C\bigl(\varepsilon+\varepsilon^\beta+\delta^{-1/2}\bigr)+o_n(1).
    \]
    Thus the maximizing sequence is bounded.
    
    \medskip
    \noindent
    \textbf{Step 2. Find the maximizer using the Radon--Nikodym property.}
     Fix \(\gamma_1>0\) and define
    \begin{multline*}
        S_n:=\big\{(x,y,t,s)\in V\times \R^d\times[0,T]^2: \|x-x_n\|^2+|y-y_n|^2
        +|t-t_n|^2+|s-s_n|^2\le \gamma_1^2 \big\}.
    \end{multline*}
    Let
    \[
        d_n:=\sup\Phi-\Phi(x_n,y_n,t_n,s_n),
    \]
    and set
    \[
        \Phi_n (x,y,t,s) = \Phi (x,y,t,s) -\frac{2d_n}{\gamma_1^2} \left(
        \|x-x_n\|^2+\|y-y_n\|^2+|t-t_n|^2+|s-s_n|^2 \right).
    \]
    By the Radon--Nikodym property, there exist
    \[
        p_n\in V,\qquad q_n\in \R^d,\qquad a_n, b_n\in\R ,
    \]
    with
    \[
        \|p_n\|+|q_n|+|a_n|+|b_n|\le \frac{d_n}{2\gamma_1},
    \]
    such that
    \[
    (x,y,t,s) \mapsto \Phi_n(x,y,t,s)-\langle p_n,x\rangle-\langle q_n,y\rangle-a_n t- b_n s
    \]
    attains its maximum over \(S_n\) at some $ (\hat x_n,\hat y_n,\hat t_n,\hat s_n)$. By the quadratic penalty in $\Phi_n$ and the bound on the perturbation, the defining localization inequality for $S_n$ is strict at this maximizer, although $\hat t_n=T$ or $\hat s_n=T$ may still occur.
    Consequently, from \eqref{eq:s4_7}--\eqref{eq:s4_8},
    \begin{equation} \label{eq:close-hat}
            |\mathfrak m(\hat x_n)-\hat y_n|+|\hat t_n-\hat s_n|
            \le C(\e ^\beta+\gamma_1) + o_n(1),
    \end{equation}
    and
    \begin{equation} \label{eq:y-hat-bound}
            |\hat y_n| \le C(\delta^{-1/2}+\gamma_1)+ o_n(1).
    \end{equation}
    We choose \(\gamma_1=o(\e ^\beta)\), so that
    \[
       \widehat P_n := \frac{\mathfrak m(\hat x_n)-\hat y_n}{\e ^\beta}
    \]
    remains uniformly bounded.

    We claim that, when $K=A\varepsilon^{1/3}$ with $A$ sufficiently large, the two time variables $\hat t_n, \hat s_n$ cannot both remain uniformly away from $0$. Suppose, to the contrary, that $\hat t_n, \hat s_n\geq\tau$ for some $\tau>0$ after passing to a subsequence.
    
    \medskip
    \noindent
    \textbf{Step 3. Separating the fast variable.}
    Fix $\alpha, \kappa >0$ and define
    \begin{align*}
        \Psi_n(x,\xi,z,t)
        :=
        &\Phi_n(x,\hat y_n,t,\hat s_n)-\langle p_n,x\rangle - a_n t\notag\\
        &+\e\Biggl[
        v^\lambda\!\left(\frac{x}{\e};
        \frac{\mathfrak m(x)-\hat y_n}{\e^\beta}\right)
        -
        v^\lambda\!\left(\xi;
        \frac{\mathfrak m(z)-\hat y_n}{\e^\beta}\right)
        \Biggr]\notag\\
        &-\frac{\|x-\e\xi\|^2+\|x-z\|^2}{2\alpha} - \kappa \left(\|x - \hat x_n\|^2 + |t - \hat t_n|^2 \right).
    \label{eq:Psi-rate}
    \end{align*}
    on a bounded neighborhood $U_n$ of \((\hat x_n,\hat x_n/\e ,\hat x_n,\hat t_n)\), chosen so that
    \[
        (x,\hat y_n,t,\hat s_n)\in S_n, \qquad t > \tau/2.
    \]
    Choose a maximizing sequence $(x_n^m,\xi_n^m,z_n^m,t_n^m) \in U_n$ such that
    \[
        \Psi_n(x_n^m,\xi_n^m,z_n^m,t_n^m) \uparrow \sup_{U_n} \Psi_n
    \qquad\text{as}\qquad m\to\infty,
    \]
    Comparison with $(x_n^m,\xi_n^m,x_n^m,t_n^m)$ and $(x_n^m,x_n^m/\e,z_n^m,t_n^m)$, with Lemma~\ref{lem4.2}, yields
    \begin{equation}\label{eq:x-z-close-rate}
        \|x_n^m-z_n^m\| \le C \alpha \e^{1-\beta-\theta} + o_m(1) \quad \text{and } \quad \|x_n^m-\e\xi_n^m\| \le C\alpha+o_m(1).
    \end{equation}
    Comparison with $(\hat x_n, \hat x_n/\e,\hat x_n, \hat t_n)$, with the maximality of $(\hat x_n, \hat y_n, \hat t_n, \hat s_n)$, yields
    \begin{equation}\label{eq:s4_12}
    \kappa \left(\|x_n^m-\hat x_n\|^2
    +
    |t_n^m-\hat t_n|^2
    \right) \leq o_m(1)+o_\alpha(1).
    \end{equation}
    By~\eqref{eq:x-z-close-rate} and \eqref{eq:s4_12}, letting first $m \to \infty$ and then $\alpha \downarrow 0$, we obtain
    \begin{equation}\label{eq:s4_13}
        (x_n^m,\xi_n^m,z_n^m,t_n^m) \to \left(\hat x_n,\hat x_n/\e,\hat x_n,\hat t_n\right).
    \end{equation}
    
    We localize once more near $(x_n^m,\xi_n^m,z_n^m,t_n^m)$. Fix $\gamma_2>0$ and set
    \[
        S_n^m
        :=
        \Bigl\{
        (x,\xi,z,t):
        \|x-x_n^m\|^2+\e^2\|\xi-\xi_n^m\|^2+\|z-z_n^m\|^2+|t-t_n^m|^2
        \le \gamma_2^2
        \Bigr\}.
    \]
    Shrinking $\gamma_2$ if necessary, we assume that $S_n^m\subset U_n$. Let
    \[
        d_n^m := \sup \Psi_n - \Psi_n(x_n^m,\xi_n^m,z_n^m,t_n^m),
    \]
    and define
    \begin{align*}
        \Psi_n^m(x,\xi,z,t) := \;& \Psi_n(x,\xi,z,t) \\
        &\quad -\frac{2d_n^m}{\gamma_2^2} \bigl( \|x-x_n^m\|^2+\e^2\|\xi-\xi_n^m\|^2+\|z-z_n^m\|^2+|t-t_n^m|^2 \bigr)
    \end{align*}
    on $S_n^m$. Again by Stegall's theorem, there exist
    \[
        p_n^m,q_n^m,w_n^m\in V,\qquad a_n^m\in\R ,
    \]
    such that
    \[
        \|p_n^m\|+\frac{\|q_n^m\|}{\e } +\|w_n^m\|+|a_n^m| \le \frac{d_n^m}{2\gamma_2},
    \]
    and
    \[
        (x,\xi,z,t)\mapsto
        \Psi_n^m(x,\xi,z,t)
        -\langle p_n^m,x\rangle
        -\langle q_n^m,\xi\rangle
        -\langle w_n^m,z\rangle
        -a_n^m t
    \]
    attains its maximum over $S_n^m$ at some $(\tilde x_n^m,\tilde \xi_n^m,\tilde z_n^m, \tilde t_n^m)$. By the quadratic penalty in $\Psi_n^m$ and the bound on the perturbation, the defining localization inequality for $S_n^m$ is strict at this maximizer, although $\widetilde t_n^m=T$ may still occur.
    
    \medskip
    \noindent
    \textbf{Step 4. The two viscosity inequalities.}
    If \(\tilde t_n^m < T\), we apply the ordinary viscosity subsolution test for \(u^\varepsilon\). If \(\tilde t_n^m=T\), we apply Lemma~\ref{lem:terminal-contact} after adding quadratic terms to the test function to make the contact strict. In either case, 
    \begin{multline}\label{eq:sub-u-e-rate}
        \frac{\tilde t_n^m-\hat s_n}{\e^\beta}
        +K+\mathcal{E}^0_{nm}
        +H\!\left(
        \frac{\tilde x_n^m}{\e},
        \frac{\mathfrak m(\tilde x_n^m)-\hat y_n}{\e^\beta}\chi_I
        +\frac{\tilde x_n^m-\e\tilde\xi_n^m}{\alpha}
        +\frac{\tilde x_n^m-\tilde z_n^m}{\alpha}
        +\mathcal{E}^1_{nm}
        \right)\le 0,
    \end{multline}
    where
    \begin{align*}
        \mathcal{E}^0_{nm}
        &:=
        \frac{4d_n}{\gamma_1^2}(\tilde t_n^m-t_n)+a_n
        +\frac{4d_n^m}{\gamma_2^2}(\tilde t_n^m-t_n^m)+a_n^m + 2\kappa (\tilde t_n^m - \hat t_n),\\ 
        \mathcal{E}^1_{nm}
        &:=
        \frac{4d_n}{\gamma_1^2}(\tilde x_n^m-x_n)+p_n
        +\frac{4d_n^m}{\gamma_2^2}(\tilde x_n^m-x_n^m)+p_n^m + 2\kap (\tilde x^m_n - \hat x_n).
    \end{align*}
    On the other hand, since $\tilde\xi_n^m$ is a local minimizer of the discounted corrector, the viscosity supersolution test for $v^\lambda$ gives
    \begin{multline}
        \lambda
        v^\lambda\!\left(
        \tilde\xi_n^m;\frac{\mathfrak m(\tilde z_n^m)-\hat y_n}{\e^\beta}
        \right)
        +H\!\left( \tilde\xi_n^m,
        \frac{\mathfrak m(\tilde z_n^m)-\hat y_n}{\e^\beta}\chi_I
        +\frac{\tilde x_n^m-\e\tilde\xi_n^m}{\alpha}
        +\mathcal{E}^2_{nm}
        \right)\ge 0,
    \label{eq:sup-vlambda-rate}
    \end{multline}
    where
    \[
        \mathcal{E}^2_{nm}
        :=
        -\frac{4d_n^m \e}{\gamma_2^2}(\tilde\xi_n^m-\xi_n^m)
        -\frac{q_n^m}{\e}.
    \]
    By construction, with \(\e \) fixed,
    \begin{equation}\label{eq:s4_16}
        \mathcal{E}^0_{nm}\to 0,\qquad \mathcal{E}^1_{nm}\to 0,\qquad \mathcal{E}^2_{nm}\to 0
    \end{equation}
    in the sequential limits
    \[
        m\to\infty, \qquad
        \gamma_2\downarrow0, \qquad
        \alpha\downarrow0, \qquad
        n\to\infty, \qquad
        \gamma_1\downarrow0, \qquad
        \delta\downarrow0.
    \]
    
    Subtracting \eqref{eq:sup-vlambda-rate} from \eqref{eq:sub-u-e-rate}, using Lemma~\ref{lem4.1}, \eqref{eq:x-z-close-rate}, and \eqref{eq:s4_16}, we obtain
    \begin{equation*}
        K+\frac{\tilde t_n^m-\hat s_n}{\e ^\beta}
        -\lambda v^\lambda\left(
        \tilde\xi_n^m;
        \frac{\mathfrak m(\tilde z_n^m)-\hat y_n}{\e ^\beta}
        \right)
        \le
        C\e ^{1-\theta-\beta}
        +o_m(1)+o_{\gamma_2}(1)+o_\alpha(1)+o_n(1).
    \end{equation*}
    Letting $m\to\infty, \gamma_2\downarrow0, \alpha\downarrow0$ sequentially, by \eqref{eq:s4_13}, we get
    \begin{equation*}
        K+\frac{\hat t_n-\hat s_n}{\e ^\beta}
        -\lambda v^\lambda\left(
        \frac{\hat x_n}{\e };
        \frac{\mathfrak m(\hat x_n)-\hat y_n}{\e ^\beta}
        \right)
        \le
        C\e ^{1-\theta-\beta}+o_n(1).
    \end{equation*}
    By \eqref{eq:vlambda-ergodic-sec4} in Lemma~\ref{lem4.2} and the boundedness of $\widehat P_n$, we obtain
    \begin{equation}\label{eq:s4_17}
        K+\frac{\hat t_n-\hat s_n}{\e ^\beta} +\overline H \left(\frac{\mathfrak m(\hat x_n)-\hat y_n}{\e ^\beta} \right)
        \le C\left(\e ^\theta+\e ^{1-\theta-\beta}\right) +o_n(1).
    \end{equation}

    \medskip
    \noindent
    \textbf{Step 5. The supersolution test for the effective limit.}
    Define
    \[
        w_n(y)
        :=
        \varepsilon
        v^\lambda\left(
        \frac{\hat x_n}{\varepsilon};
        \frac{\mathfrak m(\hat x_n)-y}{\varepsilon^\beta}
        \right).
    \]
    By Lemma~\ref{lem4.2},
    \begin{equation}\label{eq:s4_18}
        \operatorname{Lip}(w_n)
        \leq
        C\varepsilon^{1-\theta-\beta}.
    \end{equation}
    
    Set
    \begin{align*}
        G_n(y,s)
        &:=
        \frac{|\mathfrak m(\hat x_n)-y|^2
              +|\hat t_n-s|^2}
             {2\varepsilon^\beta}
        +Ks+\delta|y|^2
        \\
        &\quad
        +\frac{2d_n}{\gamma_1^2}
        \left(
        |y-y_n|^2+|s-s_n|^2
        \right)
        +\langle q_n,y\rangle+b_ns.
    \end{align*}
    Then
    \[
        (y,s)\longmapsto
        \widetilde u(y,s)+w_n(y)+G_n(y,s)
    \]
    has a local minimum at \((\hat y_n,\hat s_n)\).
    
    If $\hat s_n<T$, we apply Lemma~\ref{lem4.3} in \(\mathbb R^d\times\mathbb R\) to
    \[
        f(y,s):=w_n(y),
        \qquad
        g(y,s):=\widetilde u(y,s)+G_n(y,s).
    \]
    There exists \(r_n\in \overline D^-w_n(\hat y_n)\) such that
    \[
        (-r_n,0)
        \in
        \overline D^-\bigl(\widetilde u+G_n\bigr)
        (\hat y_n,\hat s_n).
    \]
    Since \(G_n\) is smooth,
    \[
        \left(
        -r_n-D_yG_n(\hat y_n,\hat s_n),
        -D_sG_n(\hat y_n,\hat s_n)
        \right)
        \in
        D^- \widetilde u(\hat y_n,\hat s_n).
    \]
    Moreover, by~\eqref{eq:s4_18},
    \begin{equation}\label{eq:s4_19}
        \|r_n\|
        \leq
        C\varepsilon^{1-\theta-\beta}.
    \end{equation}
    
    Using the viscosity supersolution test for $\widetilde u$, we obtain
    \begin{align*}
        0 &\leq \frac{\hat t_n-\hat s_n}{\varepsilon^\beta} -K
        -\frac{4d_n}{\gamma_1^2}(\hat s_n-s_n)
        -b_n +
        \overline H\left(
        \widehat P_n
        -2\delta\hat y_n-r_n
        -\frac{4d_n}{\gamma_1^2}(\hat y_n-y_n)
        -q_n
        \right).
    \end{align*}
   
   Using the global Lipschitz continuity of \(\overline H\), together with \eqref{eq:s4_7} and \eqref{eq:s4_19},
    \begin{equation}\label{eq:s4_20}
        -K
        +
        \frac{\hat t_n-\hat s_n}{\varepsilon^\beta}
        +
        \overline H\left(
        \frac{\mathfrak m(\hat x_n)-\hat y_n}{\varepsilon^\beta}
        \right)
        \geq
        -C\left(
        \delta^{1/2}+\varepsilon^{1-\theta-\beta}
        \right)
        -o_n(1)-o_{\gamma_1}(1).
    \end{equation}
    If $\hat s_n=T$, we first make the minimum strict by adding a positive quadratic term. For $\eta>0$, we then consider
    \[
    \widetilde u(y,s)+w_n(y)+G_n(y,s)+\frac{\eta}{T-s},
    \qquad s<T,
    \]
    as in the proof of Lemma~\ref{lem:terminal-contact}. This function has a local minimizer $(y_{n,\eta},s_{n,\eta})$ converging to $(\hat y_n,T)$. We apply Lemma~\ref{lem4.3} at this interior point. The supersolution inequality contains the additional term
    \[
    -\frac{\eta}{(T-s_{n,\eta})^2}.
    \]
    This term has the favorable sign and may be discarded. Letting $\eta\downarrow0$ therefore gives~\eqref{eq:s4_20}.
    
    \medskip
    \noindent
    \textbf{Step 6. Conclusion.} Combining \eqref{eq:s4_17} and \eqref{eq:s4_20}, we obtain
    \[
        2K \le C\left( \delta^{1/2} +\e ^\theta +\e ^{1-\theta-\beta}\right) +o_n(1)+o_{\gamma_1}(1).
    \]
    Taking successively $n\to\infty$ and $\gamma_1\downarrow0$, we obtain
    \[
        2K\le C\bigl(\delta^{1/2}
        +\varepsilon^\theta+\varepsilon^{1-\theta-\beta}\bigr).
    \]
    Choose $\theta=\beta=1/3$ and $K=A\varepsilon^{1/3}$. For sufficiently large $A>0$, choosing $\delta>0$ sufficiently small yields a contradiction. Hence, up to a subsequence, either \(\hat t_n\to0\) or \(\hat s_n\to0\). By \eqref{eq:close-hat}, 
    \begin{equation*}
        \hat t_n + \hat s_n \le C(\e^{1/3} + \gam_1) + o_n(1).
    \end{equation*}
    
    Using the Lipschitz continuity of $u^\e$ and $\widetilde u$ in time, together with \eqref{eq:close-hat} and \rm{(I1)--(I2)}, we obtain
    \begin{equation*}
        u^\e (\hat x_n,\hat t_n)-\widetilde u(\hat y_n,\hat s_n) \le C\e ^{1/3}+o_n(1)+o_{\gamma_1}(1).
    \end{equation*}
    Moreover, since $|\e  v^\lambda(\hat x_n/\e; \widehat P_n )|\le C\e ^{2/3}$,
    \[
        \Phi(\hat x_n,\hat y_n,\hat t_n,\hat s_n)\le C\e ^{1/3}+C\e^{2/3} +o_n(1)+o_{\gamma_1}(1).
    \]
    Therefore,
    \begin{equation}\label{eq:s4_21}
        \sup \Phi \le C\e^{1/3}.
    \end{equation}
    Evaluating $\Phi$ at $(x,\mathfrak m(x), t,t)$ and using \eqref{eq:s4_21}, we obtain
    \begin{align*}
        u^\varepsilon(x,t)-\widetilde u(\mathfrak m(x),t)
        &\leq
        C_T\varepsilon^{1/3}
        +
        \varepsilon v^\lambda(x/\varepsilon;0)
        +
        2Kt
        +
        \delta|\mathfrak m(x)|^2.
    \end{align*}
    By Lemma~\ref{lem4.2}, we have $|\varepsilon v^\lambda(x/\varepsilon;0)|\leq C\varepsilon^{2/3}$. Therefore, letting \(\delta\downarrow0\) gives
    \[
        u^\varepsilon(x,t)-\widetilde u(\mathfrak m(x),t)
        \leq
        C_T\varepsilon^{1/3}.
    \]
    
    The reverse inequality is obtained by applying the same argument to
    \[
    \begin{aligned}
        \Psi(x,y,t,s)
        :=\;&
        \widetilde u(y,s)-u^\e (x,t)
        +\e v^\lambda\left(
            \frac{x}{\e };
            \frac{y-\mathfrak m(x)}{\e ^\beta}
        \right)\\
        &-\frac{|\mathfrak m(x)-y|^2+|t-s|^2}{2\e ^\beta}
        -K(t+s)-\delta|y|^2.
    \end{aligned}
    \]
    Combining the two estimates proves the theorem.
\end{proof}

Theorem~\ref{thm2} completes the homogenization analysis by strengthening the qualitative convergence of Theorem~1.1 to a quantitative estimate. The rate \(O(\e ^{1/3})\) agrees with that obtained in the finite-dimensional nonconvex periodic setting, but the present argument requires additional structural ingredients: the reduction to the mean configuration, the compact quotient \(SS^d\), and the use of the Radon--Nikodym property to compensate for the lack of local compactness in \(V\). Thus, although the limiting equation is finite-dimensional, the proof of convergence is genuinely infinite-dimensional because the cell problem and the oscillatory correctors are posed on the full configuration space modulo the symmetry structure.

\section*{Acknowledgments}
The author would like to thank Professor Hung V. Tran for his invaluable feedback and suggestions. The author is also grateful to Professors Diogo Gomes, Jin Feng, and Levon Nurbekyan for their helpful comments.

\bibliographystyle{naturemag}
\bibliography{references}

@article{CL83,
  title={Viscosity solutions of Hamilton--Jacobi equations},
  author={Crandall, Michael G and Lions, Pierre-Louis},
  journal={Transactions of the American Mathematical Society},
  volume={277},
  number={1},
  pages={1--42},
  year={1983}
}

@article{CEL84,
  title={Some properties of viscosity solutions of {Hamilton--Jacobi} equations},
  author={Crandall, Michael G and Evans, Lawrence C and Lions, P-L},
  journal={Transactions of the American Mathematical Society},
  volume={282},
  number={2},
  pages={487--502},
  year={1984}
}

@book{tran2021,
  title={{Hamilton--Jacobi} equations: theory and applications},
  author={Tran, Hung V},
  volume={213},
  year={2021},
  publisher={American Mathematical Soc.}
}

@article{Evans-perturb,
  title={Periodic homogenisation of certain fully nonlinear partial differential equations},
  author={Evans, Lawrence C},
  journal={Proceedings of the Royal Society of Edinburgh Section A: Mathematics},
  volume={120},
  number={3-4},
  pages={245--265},
  year={1992},
  publisher={Royal Society of Edinburgh Scotland Foundation}
}

@article{LPV,
  title={Homogenization of {Hamilton--Jacobi} equations},
  author={Lions, Pierre-Louis and Papanicolaou, George and Varadhan, Srinivasa RS},
  journal={Unpublished preprint},
  year={1987}
}

@article{Con96,
  title={Periodic homogenization of {Hamilton--Jacobi} equations: additive eigenvalues and variational formula},
  author={Concordel, Marie C},
  journal={Indiana University Mathematics Journal},
  volume = {45},
  number = {4},
  pages={1095--1117},
  year={1996},
  publisher={JSTOR}
}

@article{Con97,
  title={Periodic homogenisation of {Hamilton--Jacobi} equations: 2. Eikonal equations},
  author={Concordel, Marie C},
  journal={Proceedings of the Royal Society of Edinburgh Section A: Mathematics},
  volume={127},
  number={4},
  pages={665--689},
  year={1997},
  publisher={Royal Society of Edinburgh Scotland Foundation}
}

@article{capuzzo-rate,
  title={On the rate of convergence in homogenization of {Hamilton--Jacobi} equations},
  author={Capuzzo-Dolcetta, Italo and Ishii, Hitoshi},
  journal={Indiana University Mathematics Journal},
  volume={50},
  number={3},
  pages={1113--1129},
  year={2001},
  publisher={JSTOR}
}

@article{han2023rate,
  title={Rate of convergence in periodic homogenization for convex {Hamilton--Jacobi} equations with multiscales},
  author={Han, Yuxi and Jang, Jiwoong},
  journal={Nonlinearity},
  volume={36},
  number={10},
  pages={5279--5297},
  year={2023},
  publisher={IOP Publishing}
}

@article{TY,
    author = "Hung V Tran and Yifeng Yu",
     title = "Optimal convergence rate for periodic homogenization of convex {Hamilton--Jacobi} equations",
   journal = "Indiana Univ. Math. J.",
  fjournal = "Indiana University Mathematics Journal",
    volume = 74,
      year = 2025,
     issue = 3,
     pages = "555--573",
      issn = "0022-2518",
     coden = "IUMJAB",
   mrclass = "35B10, 35B27, 35B40, 35F21, 49L25",
}

@article{QTY,
  title={Min--max formulas and other properties of certain classes of nonconvex effective {Hamiltonians}},
  author={Qian, Jianliang and Tran, Hung V and Yu, Yifeng},
  journal={Mathematische Annalen},
  volume={372},
  number={1},
  pages={91--123},
  year={2018},
  publisher={Springer}
}

@article{mitake2019rate,
  title={Rate of convergence in periodic homogenization of {Hamilton--Jacobi} equations: the convex setting},
  author={Mitake, Hiroyoshi and Tran, Hung V and Yu, Yifeng},
  journal={Archive for Rational Mechanics and Analysis},
  volume={233},
  number={2},
  pages={901--934},
  year={2019},
  publisher={Springer}
}

@article{C-L-1,
  title={{Hamilton--Jacobi} equations in infinite dimensions I. Uniqueness of viscosity solutions},
  author={Crandall, Michael G and Lions, Pierre-Louis},
  journal={Journal of Functional Analysis},
  volume={62},
  number={3},
  pages={379--396},
  year={1985},
  publisher={Elsevier}
}

@article{C-L-2,
  title={{Hamilton--Jacobi} equations in infinite dimensions. II. Existence of viscosity solutions},
  author={Crandall, Michael G and Lions, Pierre-Louis},
  journal={Journal of Functional Analysis},
  volume={65},
  number={3},
  pages={368--405},
  year={1986},
  publisher={Academic Press}
}

@article{Ishii,
  author  = {Ishii, Hitoshi},
  title   = {Perron's method for {Hamilton--Jacobi} equations},
  journal = {Duke Mathematical Journal},
  volume  = {55},
  number  = {2},
  pages   = {369--384},
  year    = {1987},
  doi     = {10.1215/S0012-7094-87-05521-9}
}

@article{barbu,
  title={{Hamilton--Jacobi} equations in {Hilbert} spaces},
  author={Barbu, Viorel and Da Prato, Giuseppe},
  journal={Pitman, London},
  year={1983}
}

@article{cannarsa,
  title={Some results on non-linear optimal control problems and {Hamilton--Jacobi} equations in infinite dimensions},
  author={Cannarsa, P and Da Prato, G},
  journal={Journal of Functional Analysis},
  volume={90},
  number={1},
  pages={27--47},
  year={1990},
  publisher={Elsevier}
}

@article{stegall1978,
  title={Optimization of functions on certain subsets of {Banach} spaces},
  author={Stegall, Charles},
  journal={Mathematische Annalen},
  volume={236},
  number={2},
  pages={171--176},
  year={1978},
  publisher={Springer}
}

@article{Gomes-Nurbekyan,
  title={An infinite-dimensional weak {KAM} theory via random variables},
  author={Gomes, Diogo and Nurbekyan, Levon},
  journal={Discrete and Continuous Dynamical Systems},
  volume={36},
  number={11},
  pages={6167--6185},
  year={2016}
}

@article{Gangbo-1,
  title={Lagrangian dynamics on an infinite-dimensional torus; a weak {KAM} theorem},
  author={Gangbo, Wilfrid and Tudorascu, Adrian},
  journal={Advances in Mathematics},
  volume={224},
  number={1},
  pages={260--292},
  year={2010},
  publisher={Elsevier}
}

@article{Gangbo-2,
  title={A weak {KAM} theorem; from finite to infinite dimension},
  author={Gangbo, Wilfrid and Tudorascu, Adrian},
  journal={Optimal transportation, geometry and functional inequalities},
  volume={11},
  pages={45--72},
  year={2010}
}

@article{feng,
  title={On a {Hamilton--Jacobi} {PDE} theory for hydrodynamic limit of action minimizing collective dynamics},
  author={Feng, Jin},
  journal={arXiv preprint arXiv:2512.20809},
  year={2025}
}

@article{fathi2003,
  title={Weak {KAM} theorem in {Lagrangian} dynamics},
  author={Fathi, Albert},
  year={2003},
  publisher={Cambridge University Press Cambridge}
}

@article{padhye1996fluid,
  title={Fluid element relabeling symmetry},
  author={Padhye, Nikhil and Morrison, Philip J},
  journal={Physics Letters A},
  volume={219},
  number={5-6},
  pages={287--292},
  year={1996},
  publisher={Elsevier}
}

@article{gomes2015,
  title={On the minimizers of calculus of variations problems in {Hilbert} spaces},
  author={Gomes, Diogo and Nurbekyan, Levon},
  journal={Calculus of Variations and Partial Differential Equations},
  volume={52},
  number={1},
  pages={65--93},
  year={2015},
  publisher={Springer}
}

@article{Ding-Ekren-Han-Zitridis,
  title={Quantitative homogenization of convex {Hamilton--Jacobi} equations in the {Wasserstein} space},
  author={Ding, Zhiyan and Ekren, Ibrahim and Han, Yuxi and Zitridis, Antonios},
  journal={arXiv preprint arXiv:2606.22103},
  year={2026}
}
\end{document}